\documentclass[12pt,a4paper]{amsart}
\usepackage[abbrev]{amsrefs}
\usepackage{amssymb}
\usepackage{bbm}

\theoremstyle{plain}
 \newtheorem{theorem}{Theorem}[section]
 \newtheorem{lemma}[theorem]{Lemma}
 \newtheorem{corollary}[theorem]{Corollary}
 \newtheorem{proposition}[theorem]{Proposition}
 
\theoremstyle{definition}
 \newtheorem{definition}[theorem]{Definition}
 
\theoremstyle{remark}
 \newtheorem{remark}[theorem]{Remark}
 \newtheorem*{ack}{Acknowledgment}

\def\vol{\mathop{\mathrm{vol}}\nolimits}
\numberwithin{equation}{section}
\begin{document}
\title[]{Spectral convergence of high-dimensional spheres \\ to Gaussian spaces}
\author[]{Asuka Takatsu}
\address{Department of Mathematical Sciences, Tokyo Metropolitan University, Tokyo {192-0397}, Japan \&
 RIKEN Center for Advanced Intelligence Project (AIP), Tokyo Japan.}
\email{asuka@tmu.ac.jp}
\date{\today}
%%%
%%%
%%%%%%%%%%%%%%%%%%%%%%%%%%%%%%%%%%%%%%%%
\maketitle
\vspace{-17pt}
\begin{abstract}
We prove that the spectral structure on the $N$-dimensional standard sphere of radius $(N-1)^{1/2}$
compatible with a projection onto the first $n$-coordinates converges to the spectral structure on the $n$-dimensional Gaussian space with variance $1$ as $N\to \infty$.
We also show the analogue for the first Dirichlet eigenvalue problem on a ball in the sphere and {that} on a half-space in the Gaussian space.
\end{abstract}
%%%%%%%%%%%%%%%%%%%%%%%%%%%%%%%%%%%%%%%%%%%%%%%%%%
%%%%%%%%%%%%%%%%%%%%%%%%%%%%%%%%%%%%%%%%%%%%%%%%%%
%%%%%%%%%%%%%%%%%%%%%%%%%%%%%%%%%%%%%%%%%%%%%%%%%%
%%%%%%%%%%%%%%%%%%%%%%%%%%%%%%%%%%%%%%%%%%%%%%%%%%
%%%%%%%%%%%%%%%%%%%%%%%%%%%%%%%%%%%%%%%%%%%%%%%%%%
\section{Introduction}
A curvature-dimension condition $\mathrm{CD}(\kappa, N)$ imposes restriction on the spectra of the weighted Laplacian on a weighted manifold.
For example, the Lichnerowicz--Obata type eigenvalue estimate is known 
 (see \cite{CZ}*{Theorems 1,2}, \cite{Ke}*{Corollary~1.3}, \cite{LV}*{Theorem~5.34} and the references therein).
Here a \emph{weighted manifold} $(M,\mu)$ is a complete smooth $n$-dimensional Riemannian manifold $(M,g)$ equipped with a measure $\mu$ of the form
\[
\mu=\exp(-\Psi) \vol_M, 
\]
where $\Psi\in C^\infty(M)$ and $\vol_M$ denotes the Riemannian volume measure on $(M,g)$.
The \emph{weighted Laplacian} $\Delta_\mu$ on $(M,\mu)$ is defined as
\[
\Delta_{\mu} f:=\Delta_M f-g(\nabla_M \Psi, \nabla_M f )\qquad \text{for\ }f \in C^\infty(M), 
\]
where $\nabla_M$ and $\Delta_M$ stand for the gradient and the Laplacian on $(M,g)$, respectively,
so that the following integration by parts is satisfied
\[
\int_{M} g(\nabla_M f_1, \nabla_M f_2)d\mu=-\int_{M} f_1 \Delta_{\mu} f_2 d\mu
\qquad\text{for\ } f_1, f_2\in C^\infty_0(M).
\]
Given $\kappa \in \mathbb{R}$ and $N\in [n,\infty]$, 
we say that $(M,\mu)$ satisfies the \emph{curvature-dimension condition $\mathrm{CD}(\kappa, N)$} if 
\[
\mathrm{Ric}_M (v,v)+\mathrm{Hess}_M\Psi(v,v) -\frac{v(\Psi)^2}{N-n} \geq \kappa g(v,v) \qquad\text{for\ }v\in TM, 
\]
where $\mathrm{Ric}_M$ is the Ricci curvature tensor and $\mathrm{Hess}_M $ is the Hessian operator on $(M,g)$, respectively.
To make sense, we employ the convention that $\frac{1}{\infty}:=0, \frac{1}{0}:=+\infty$,
and $\infty \cdot 0:=0$.
A model space for comparison geometry under the condition $\mathrm{CD}(1, N)$ is 
the $N$-dimensional standard sphere of radius $(N-1)^{1/2}$ for $N\in \mathbb{N}$ with $N\geq2$, 
and the one-dimensional Gaussian space with variance $1$ for $N=\infty$.
%

%%%
For $N\in \mathbb{N}$ and $a>0$, let $\mathbb{S}^N(a)$ be the $N$-dimensional standard sphere of radius $a$.
%
%%%%%%%%%%%%%%%%%%%%%%%%%%%%%%%%%%%%%%%%%%%%%%%%%%%%%%%%%%%%
We denote by $\langle \cdot ,\cdot \rangle$ the Euclidean inner product and set $|\cdot|_2:=\langle\cdot, \cdot \rangle^{1/2}$.
For $n\in \mathbb{N}$ and $\alpha>0$,
we denote by $\gamma^n_{\alpha}$ the \emph{$n$-dimensional Gaussian measure with variance $\alpha^2$}, that is, 
\[
d\gamma^n_{\alpha}(x)=(2\pi \alpha^2)^{-\frac{n}{2}} \exp\left(-\frac{|x|_2^2}{2\alpha^2} \right)dx.
\]
The weighted manifold $\Gamma^n_{\alpha}:=(\mathbb{R}^n, \gamma^n_{\alpha})$ is called the \emph{$n$-dimensional Gaussian space with variance $\alpha^2$}.
Notice that 
a weighted manifold of $\mathbb{S}^N(a)$ equipped with its Riemannian volume measure satisfies $\mathrm{CD}( a^{-2}(N-1), N)$ 
and $\Gamma^n_{\alpha}$ satisfies $\mathrm{CD}(\alpha^{-2}, \infty)$, respectively.
Set
\[
S_N:=\mathbb{S}^N(\sqrt{N-1}), \qquad \gamma^n:=\gamma^n_1,\qquad \Gamma^n:=\Gamma^n_1.
\]

%%%
Since $\mathrm{CD}(1, \infty)$ can be regarded as the limit of $\mathrm{CD}(1, N)$ as $N\to \infty$,
the spectral structure on $\Gamma^n$ would be derived from the asymptotic behavior of that on $S_N$ as well.
For example, Borell~\cite{Bo}*{Theorem~3.1} and Sudakov--Cirel\cprime son~\cite{SC}*{Corollary~1} independently proved the Brunn--Minkowski inequality on $\Gamma^n$ by using that on $S_N$.
The Brunn--Minkowski inequality determines a domain minimizing the first Dirichlet eigenvalue under the restriction of the volume.
The key of the proof is the following asymptotic behavior, so-called Poincar\'e's theorem
(we refer to \cite{DF}*{Section 6} for the history of Poincar\'e's theorem).
Let $\sigma_N$ be the normalized Riemannian volume measure on $S_N$ to be a probability measure.
For $n,N\in \mathbb{N}$ with $n\leq N$, $p^N_n$ denotes the projection from $\mathbb{R}^{N+1}=\mathbb{R}^{n} \times \mathbb{R}^{N-n+1}$ onto $\mathbb{R}^n$ defined by 
\[
p^N_n(x,y):=x\qquad \text{for}\ (x,y) \in \mathbb{R}^n \times \mathbb{R}^{N-n+1}.
\]
Then the push-forward measure of $\sigma_{N}$ by the restriction of $p^N_n$ to $S_N$ satisfies 
\[
\lim_{N \to \infty} \frac{d \left( (p^N_n|_{S_N})_\sharp\sigma_{N}\right)}{dx}(x)=\frac{d \gamma^n}{dx}(x)
\qquad\text{for\ } x\in \mathbb{R}^n
\]
and $\{(p^N_n|_{S_N})_\sharp \sigma_{N}\}_{N\in \mathbb{N}}$ converges to $\gamma^n$ weakly as $N\to \infty$.
Since the weak convergence of probability measures on $\mathbb{R}^n$ is metrizable by the Prokhorov metric~$d_P$, 
%%%
it holds that 
\[
\lim_{N\to \infty} d_P((p^N_n|_{S_N})_\sharp \sigma_{N}, \gamma^{n})=0.
\]
%
%%%%%%%%%%%%%%%%%%%%%%%%%%%%%%%%%%%%%%%%%%%%%%%%%%

%%%%%%%%%%%%%%%%%%%%%%%%%%%%%%%%%%%%%%%%%%%%%%%%%%
Let $\iota_N :S_N \hookrightarrow \mathbb{R}^{N+1}$ be the inclusion map.
In contrast to Poincar\'e's theorem, Shioya and the author \cite{ST}*{Theorem 1.4} showed that
\[
\liminf_{N\to \infty} d_P( \iota_{N}{}_\sharp \sigma_{N}, \gamma^{N+1})>0.
\]
%%%%%
This suggests that the asymptotic behavior of the spectral structure on $S_N$ and $\Gamma^{N+1}$ are different.
Indeed, the multiplicity of the first nonzero eigenvalue on both of $S_N$ and $\Gamma^{N+1}$ are $N+1$,
while the multiplicity of the second nonzero eigenvalue on $S_N$ is $N(N+3)/2$ but that on $\Gamma^{N+1}$ is $(N+1)(N+2)/2$.
See \cite{Mil2018}*{Sections 2.1, 2.2} for instance. 
%%%%%%%%%%%%%%%%%%%%%%%%%%%%%%%%%%%%%%%%%%%%%%%%%%
Thus it is more appropriate to compare the spectral structure on $\Gamma^n$
with the compatible spectral structure on $S_N$ with~$p^N_n$, rather than the spectral structure on $S_N$ itself.
%%%

In this paper, we prove the convergence of eigenvalues on $\mathbb{S}^N(a_N)$ to those on $\Gamma^n_{\alpha}$
together with the convergence of the composition of $p^N_n$ and compatible eigenfunctions on $\mathbb{S}^N(a_N)$ 
to eigenfunctions on $\Gamma^n_{\alpha}$ as $N\to \infty$ 
when $\{a_N/\sqrt{N-1}\}_{N\geq n,2}$ converges to $\alpha$.
We also show the analogue for the first Dirichlet eigenvalue {problem} on a ball in~$\mathbb{S}^N(a_N)$ and {that} on a half-space in~$\Gamma^n_{\alpha}$.

We define some {notation} needed to state our theorems.
Let $\mathbb{N}_0$ denote the set of nonnegative integers.
Unless specified otherwise in this paper, let
\[
n,N\in \mathbb{N}\quad \text{with}\quad n,2\leq N, \qquad 
k\in \mathbb{N}_0, \qquad 
a,\alpha>0, \qquad
\theta \in (0,\pi),\qquad R\in \mathbb{R}.
\]
We shall for convenience denote a sequence $\{c_N\}_{N\geq N_0}$ by $\{c_N\}_{N}$.

For the rest of this paper, a weighted manifold $(M, \mu)$ is either $\mathbb{S}^N(a)$ equipped with its Riemannian volume measure $\vol_{\mathbb{S}^N(a)}$
or $\Gamma^n_{\alpha}=(\mathbb{R}^n, \gamma^n_{\alpha})$.
Note that $\Delta_{\vol_{\mathbb{S}^N(a)}}=\Delta_{\mathbb{S}^N(a)}$.
When it will introduce no confusion, we shall denote $(\mathbb{S}^N(a),\vol_{\mathbb{S}^N(a)})$ simply by $\mathbb{S}^N(a)$.
%

%%%%%%%
A real number $\lambda$ is called a \emph{closed eigenvalue}, or simply \emph{eigenvalue} of $-\Delta_\mu$ on $M$ 
if there exists a nontrivial solution $\phi\in C^2(M)$ to
\begin{equation}\label{EP}
\Delta_\mu \phi=-\lambda \phi \qquad \text{in $M$}.
\end{equation}
A solution to \eqref{EP} is called an \emph{eigenfunction} of eigenvalue $\lambda$.
Any constant function on $M$ is an eigenfunction of eigenvalue $0$.
We denote the list of distinct eigenvalues on $M$ by
\[
0=\lambda_0(M,\mu)<\lambda_1(M,\mu)<\lambda_2(M,\mu)<\cdots <\lambda_k(M,\mu)<\cdots \uparrow \infty.
\]
%%%
Let $E_{k}(M,\mu)$ be the linear space of solutions to \eqref{EP} for $\lambda=\lambda_k(M,\mu)$.
It is known that the linear space $E_{k}(\mathbb{S}^N(a))$ is spanned 
by the restriction of homogeneous harmonic polynomials on $\mathbb{R}^{N+1}$ of degree $k$ to $\mathbb{S}^N(a)$
(see \cite{Ch}*{Section II.4}).
We denote by $\mathbb{P}(n)$ the linear space of polynomials on $\mathbb{R}^n$.
Define the linear subspace of $E_{k}(\mathbb{S}^N(a))$ by
\begin{align*}
E_{k}^{n}(\mathbb{S}^N(a)):=
\left\{ \Phi\in E_{k}(\mathbb{S}^N(a)) \biggm| 
\begin{tabular}{l}
$Q \circ p^N_n=\Phi$ on $\mathbb{S}^N(a)$ for some $Q\in \mathbb{P}(n)$
\end{tabular}
\right\}.
\end{align*}
%
%%%%%%%%%%%%%%%%%%%%%%%%%%%%%%%
%%%%%%%%%%%%%%%%%%%%%%%%%%%%%%%%%%%%%%%%%%%%%%%
%%%%%%%%%%%%%%%%%%%%
\begin{theorem}\label{main1}
Let $\{a_N\}_{N}$ be a sequence of positive real numbers.
For $n, N\in \mathbb{N}$ with $n,2\leq N$ and $k\in\mathbb{N}_0$, 
\begin{equation}\label{m1}
\dim E_{k}^{n}(\mathbb{S}^N(a_N))=\dim E_{k}(\Gamma^n)=:d_k(n).
\end{equation} 
Moreover, if $\{a_N/\sqrt{N-1}\}_{N}$ converges to a positive real number $\alpha$ as $N\to \infty$, then 
\begin{equation}\label{m2}
\lim_{N \to \infty } \lambda_k(\mathbb{S}^N(a_N))=\lambda_k(\Gamma^n_{\alpha}).
\end{equation}
In this case, there exist 
a set of  homogeneous harmonic polynomials $\{P_{N,j}\}_{j=1}^{d_k(n)}$ on $\mathbb{R}^{N+1}$ of degree $k$ 
and $\{Q_{N,j}\}_{j=1}^{d_k(n)}\subset \mathbb{P}(n)$ satisfying the following three properties$:$
\begin{itemize}
\item
the restriction of  $\{P_{N,j}\}_{j=1}^{d_k(n)}$to $\mathbb{S}^N(a)$ forms a basis of $E_{k}^{n}(\mathbb{S}^N(a_N))$.
\item
$Q_{N,j} \circ p^N_n=P_{N,j}$ on $\mathbb{S}^N(a_N)$.
\item
$\{Q_{N,j}\}_{N}$ converges to some $Q_j\in \mathbb{P}(n)$ uniformly on compact sets and strongly in~$L^2(\Gamma^n_{\alpha})$ as $N\to \infty$ for each $1\leq j\leq d_{k}(n)$
and $\{Q_j\}_{j=1}^{d_k(n)}$ forms a basis of~$E_{k}(\Gamma^n_{\alpha})$.
\end{itemize}
\end{theorem} 
%%%
%%%%%%%%%%%%%%%%%%%%%%%%%%%%%%%%%%%%%%%%%%%%%%%%%%

%%%%%%%%%%%%%%%%%%%%%%%%%%%%%%%%%%%%%%%%%%%%%%%%%%
Next we consider the analogue for the first Dirichlet eigenvalue problem.
For $m,i\in \mathbb{N}$ with $i\leq m$,
let $e_i^m$ denote the $m$-tuple consisting of zeros except for a $1$ in the $i$th spot.
Let $d_{\mathbb{S}^N(a)}$ be the Riemannian distance function on $\mathbb{S}^N(a)$.
We define the open ball $B^N_{a\theta}$ in $\mathbb{S}^N(a)$ and the open half-space $V_{\alpha R}^n$ in $\mathbb{R}^n$ by 
\begin{align*}
B^N_{a\theta}:=
\left\{ z \in \mathbb{S}^N(a) \ |\ d_{\mathbb{S}^N(a)}\left(z, a e_1^{N+1}\right) < a\theta \right\},\qquad
V_{\alpha R}^n:=\{ x=(x_i)_{i=1}^n \in \mathbb{R}^n \ |\ x_1> \alpha R \},
\end{align*}
respectively.
Let $\Omega=B^N_{a\theta}$ if $M=\mathbb{S}^N(a)$, and $\Omega=V_{\alpha R}^n$ if $M=\Gamma^n_{\alpha}$.
%%%

%%%%%%%
A real number $\lambda$ is called the \emph{first Dirichlet eigenvalue} of $-\Delta_\mu$ on $\Omega$
if there exists a solution $\phi \in C^2(\Omega)\cap C^0(\overline{\Omega})$ to
\begin{equation}\label{DEP}
\begin{cases}
\Delta_\mu \phi=-\lambda \phi & \text{in\ $\Omega$},\\
\phi>0 & \text{in\ $\Omega$},\\
\phi=0 & \text{on\ $\partial\Omega$}.
\end{cases}
\end{equation}
The first Dirichlet eigenvalue of $-\Delta_\mu$ on $\Omega$, denoted by $\lambda(\Omega, (M,\mu))$, is positive 
and a solution to~\eqref{DEP} is uniquely determined up to a positive constant multiple.
A solution to~\eqref{DEP} is called a \emph{first positive Dirichlet eigenfunction} of $-\Delta_\mu$ on $\Omega$.
%%%

Let $H^1_0(V_{\alpha R}^n,\gamma^n_{\alpha})$ denote the completion of $C_0^\infty(V_{\alpha R}^n)$ with respect to the inner product given by
\[
(f_1, f_2)_{H^1(V_{\alpha R}^n,\gamma^n_{\alpha})}:=\int_{V_{\alpha R}^n} f_1f_2 d\gamma^n_{\alpha}
+\int_{V_{\alpha R}^n} \langle\nabla_{\mathbb{R}^n} f_1,\nabla_{\mathbb{R}^n} f_2 \rangle d\gamma^n_{\alpha}
\qquad\text{for\ } f_1, f_2\in C^\infty_0(V_{\alpha R}^n).
\]

\begin{theorem}\label{main2}
Let $\{a_N\}_{N}$, $\{\theta_N\}_{N}$ be sequences of real numbers such that 
$a_N>0$ and $\theta_N\in (0,\pi)$ for $N \in \mathbb{N}$.
Define two functions $s_N, w_N$ on $[-a_N, a_N]$ and a function $w_\infty$ on $\mathbb{R}$ by 
\begin{align*}
s_N(r):=1-\frac{r^2}{a_N^2},\quad
w_N(r)
:={s_N(r)^{\frac{N}{2}-1}}\cdot\left({\displaystyle \int_{-a_N}^{a_N} s_N(\rho)^{\frac{N}{2}-1}d\rho}\right)^{-1},\quad
w_\infty(r):= \frac{1}{\sqrt{2\pi}\alpha}e^{-\frac{r^2}{2\alpha^2}},
\end{align*}
respectively.
Let $\phi_N$ be the first positive Dirichlet eigenfunction of $-\Delta_{\mathbb{S}^N(a_N)}$ on $B^N_{a_N \theta_N}$ such that 
\[
\int_{B_{a_N \theta_N}^N} \phi_N(z)^2 d\!\vol_{\mathbb{S}^N(a_N)}(z)=
\vol_{\mathbb{S}^{N-1}(a_N)}(\mathbb{S}^{N-1}(a_N)) \int_{-a_N}^{a_N} s_N(r)^{\frac{N}{2}-1}dr.
\]
Then for $n, N\in \mathbb{N}$ with $n,2\leq N$, 
there exists $\psi_N \in H_0^1(V_{\alpha R}^n, \gamma^n_{\alpha})$ such that
\begin{equation}\label{dproj}
\psi_N \circ p^N_n =\phi_N \cdot \left\{\left(s_{N}\sqrt{\frac{w_N}{w_\infty}}\right) \circ p^N_1\right\}
\qquad
\text{on\ }B^N_{a_N \theta_N}.
\end{equation}
Moreover, if there exist $\alpha>0$ and $R\in \mathbb{R}$ such that 
\begin{align*}%\begin{split} \label{bounded} 
%%%%
\lim_{N\to \infty} \frac{a_N}{\sqrt{N-1}}=\alpha, \qquad\qquad \qquad\qquad
&\lim_{N\to \infty} a_N \cos \theta_N=\alpha R, \\
\sup_{N\in \mathbb{N}}\frac{a_N^2 -\alpha^2(N-2)}{a_N}<\infty,
\ \qquad\qquad&
\ \qquad a_N \cos \theta_N\geq \alpha R,
%\end{split}
\end{align*}
then 
\begin{align*}
%%%%%%
\lim_{N \to \infty} \lambda(B^N_{a_N \theta_N}, \mathbb{S}^N(a_N))=\lambda(V_{\alpha R}^n, \Gamma^n_{\alpha}).
%%%
\end{align*}
In this case, $\{\psi_N\}_{N}$ converges to the first positive Dirichlet eigenfunction $\psi_\infty$ of $-\Delta_{\gamma_\alpha^n}$ on $V_{\alpha R}^n$ 
strongly in $H^1_0(V_{\alpha R}^n, \gamma^n_{\alpha})$ and
\[
\int_{V_{\alpha R}^n} \psi_\infty(x)^2 d\gamma^n_{\alpha}(x)=1.
\]
\end{theorem}
%%%%%%%%%%%%%%%%%%%%%%%%%%%%%%%%%%%%%%%

Let us make a few comments on related works.
Aside from the difference between the asymptotic behavior of the spectral structure on $S_N$ and $\Gamma^{N+1}$,
the study of the relation between the limit of $S_N$ as $N\to\infty$ and  the infinite-dimensional Gaussian space has a long history, 
which  goes back to  Boltzmann and  Maxwell around the 1860s in the  study of the motion of gas molecules.
McKean~\cite{Mc} gave an exposition to explain how this study is fruitful
(see also~\cite{Hidabook}, where the classical idea of  L\'evy~\cite{Le} and Wiener~\cite{Wi} is explained with examples in physics and control theory).
Its mathematical foundations are established in the 1960s.
For example, Hida--Nomoto~\cite{HN} constructed an infinite-dimensional Gaussian space as the projective limit space of $S_N$
and defined a  family of functions analogous to homogeneous harmonic polynomials restricted to $S_N$, which  forms a complete orthonormal system in the $L^2$-spaces on the  infinite-dimensional Gaussian space. 
Umemura--K\^ono~\cite{UK}*{Section~4} make clear the relation between the Laplacian on $S_N$ and that on the infinite-dimensional Gaussian space
and  investigated how this relation reflects on their eigenfunctions.
Peterson--Sengupta~\cite{PS}*{Section~5} analyzed an asymptotic behavior of the Laplacian on $S_N$ and its eigenfunctions  from the algebraic viewpoint.
Compare Theorem~\ref{main1} with \cite{UK}*{Proposition~5} and \cite{PS}*{Proposition~4.3}.
Note that the difference of eigenvalues on $S_N$ and $\Gamma^1$ provides an quantitative estimate of  the difference between $S_N$ and $\Gamma^1$ by~\cite{BF}*{Theorem~1.2}.

As for the Dirichlet eigenvalue problem,
Friedland and Hayman~\cite{FH}*{Theorem~2} proved that the positive root $\nu_N(s)$ of the equation
\[
\nu(\nu+N-1)
=
\lambda(B^N_{\theta_N}, \mathbb{S}^N(1))
\qquad
\text{with}\quad
\vol_{\mathbb{S}^N(1)}(B^N_{\theta_N})/\vol_{\mathbb{S}^N(1)}(\mathbb{S}^N(1))=s\in (0,1)
\]
is nonincreasing in $N\in \mathbb{N}$ hence the limit of $\{\nu_N(s)\}_{N}$ as $N\to \infty$ exists. 
This suggests that $\{\lambda(B^N_{\theta_N}, \mathbb{S}^N(1))/N\}_{N}$ converges to $\lambda(V_{R}^1, \Gamma^1)$ as 
$N\to \infty$
(see~\cite{CS}*{p.218}).
In general, the spectral convergence with respect to the pointed measured Gromov--Hausdorff topology
under the curvature-dimension condition is known (for instance, see \cite{AH2, AH, GMS, ZZ} and the references therein).
With respect to the pointed measured Gromov--Hausdorff topology,
although $\{\mathbb{S}^N(1)\}_{{N}}$ diverges (see~\cite{Funano}*{Proposition~1.1}),
$\{(p^N_n(S_N), |\cdot|_2, (p^N_n|_{S_N})_\sharp \sigma_{N})\}_{{N}}$ converges to $\Gamma^n$ as \text{$N \to \infty$}. 
This with the metric contraction principle (see~\cite{Mil2018}*{Proposition~3.4}) suggests
\[
\lambda_m(V_{\alpha R}^n, \Gamma^n_{\alpha})
\geq
\lim_{N \to \infty} \lambda_m(B^N_{a_N \theta_N}, \mathbb{S}^N(a_N)),
\]
where $m\in \mathbb{N}$ and $\lambda_m(\Omega, (M,\mu))$ stands for the $m$th Dirichlet eigenvalue of $-\Delta_\mu$ on $\Omega$.
In the case $n=1$, 
Kazukawa~\cite{Kazukawa}*{Example~4.18} used a projection from $S_N$ to $\mathbb{R}$ different from $p^N_1$
and discussed the spectral convergence on $S_N$ in the framework of metric measure foliation.
%%%

%%%%
This paper is organized as follows. 
Section 2 is devoted to recalling some known facts of Eigenvalue problems on spheres and Gaussian spaces.
We prove Theorem \ref{main1} in Section~3 and Theorem \ref{main2} in Section~4, respectively. 
We discuss the relation between Dirichlet eigenspaces of high-dimensional spheres and those of Gaussian spaces in Section~5.
%%%%%%%%%%%%%%%%%%%%%%%%%%%%%%%%%%%%%%%%%%%%%%%%%%%%%%%%%%%%

%%%%%%%%%%%%%%%%%%%%%%%%%%%%
\section{Eigenvalue problems on $\mathbb{S}^N(a)$ and $\Gamma^n_{\alpha}$}\label{eigen}
%%%
Let us briefly recall some known facts of eigenvalue problems on $\mathbb{S}^N(a)$ and $\Gamma^n_{\alpha}$.
We refer to \cite{Ch}*{Sections II.4, II.5} and \cite{Mil2018}*{Sections 2.1, 2.2} for more details.

%%%%%%%%%%%%%%%%%%%%%%%%%%%%
%%%%%%%%%%%%%%%%%%%%%%%%%%%%
%%%%%%%%%%%%%%%%%%%%%%%%%%%%%%%%%%%%%%%%%%%%%%%%%%%%%%%%%%%%
%%%%%%%%%%%%%%%%%%%%%%%%%%%%%%%%%%%%%%%%%%%%%%%%%%%%%%%%%%%%
%%
\subsection{Eigenvalue problem on $\mathbb{S}^N(a)$}\label{ES}
The $k$th distinct eigenvalue on $\mathbb{S}^N(a)$ is given by
\begin{equation}\label{sphere}
\lambda_k(\mathbb{S}^N(a))=\frac{k}{a^2}(k+N-1)
\qquad\text{with multiplicity}\qquad
\binom{N+k}{k}-\binom{N+k-2}{k-2},
\end{equation}
where we adhere to the convention that $\binom{N-2}{-2}, \binom{N-1}{-1}:=0$.
%
%%%%
%

For a first positive Dirichlet eigenfunction $\phi$ of $-\Delta_{\mathbb{S}^N(a)}$ on $B^N_{a\theta}$, 
there exists a solution $\varphi \in C^\infty([0, a\theta)) \cap C([0,a\theta])$ to
\[
\tag{D$^N$}
\begin{cases}
\displaystyle
\varphi''(\vartheta)+ (N-1)\frac{\cos \left(\vartheta/{a}\right)}{a\sin\left(\vartheta/{a}\right)} \varphi'(\vartheta)=-\lambda(B^N_{a\theta}, \mathbb{S}^N(a)) \varphi(\vartheta)
& \text{in\ $\vartheta\in [0,a\theta)$},\\
\varphi(\vartheta)>0& \text{in\ $\vartheta\in [0,a\theta)$},\\
\varphi(a\theta)=0, &
\end{cases} 
\]
such that $\phi(z)=\varphi( d_{\mathbb{S}^N(a)}(z,a e_1^{N+1}))$ on $\overline{B^N_{a\theta}}$,
where 
\begin{equation}\label{int}
\int_{B^N_{a\theta}} \phi(z)^2 d \vol_{\mathbb{S}^N(a)}(z)
=
\vol_{\mathbb{S}^{N-1}(a)}(\mathbb{S}^{N-1}(a)) \int_{0}^{a\theta} \varphi(\vartheta)^2
 \sin^{N-1}\left({\vartheta}/{a}\right)d\vartheta<\infty
\end{equation}
holds.
%%%%%%%%
It follows that 
\[
\lambda(B^N_{a\theta}, \mathbb{S}^N(a))=\frac{1}{a^2}\lambda(B^N_{\theta}, \mathbb{S}^N(1)).
\]
It is known that $\phi(z)=\cos ( d_{\mathbb{S}^N(a)}(z,a e_1^{N+1})/a )$ is a first positive Dirichlet eigenfunction 
of $-\Delta_{\mathbb{S}^N(a)}$ on $B^N_{a\pi/2}$ and $\lambda(B^N_{a\pi/2}, \mathbb{S}^N(a))=N/a^2$.
Hence $\varphi(r)=\cos(r/a)$ solves (D$^N$) with the case $\theta=\pi/2$.
Notice that
\[
d_{\mathbb{S}^N(a)}(z,a e_1^{N+1})=a \cdot \arccos\left(\frac{z_1}{a}\right)
\qquad
\text{on\ }z=(z_i)_{i=1}^{N+1}\in \mathbb{S}^N(a)\subset \mathbb{R}^{N+1}.
\]
%%

%%%%%%%%%%%%%%%%%%%%%%%%%%%%%%%%%%%%%%%%%%%%%%%%%%%%%%%%%%%%
%%%%%%%%%%%%%%%%%%%%%%%%%%%%%%
\subsection{Eigenvalue problem on $\Gamma^n_{\alpha}$}\label{EG}
%%%%%%%%%%%%%%%%%%%%%%%%%%%%%%
%%%%%%%%%%%%%%%%%%%%%%%%%%%%%%%%%%%%%%%%%%%%%%%%%%%%%%%%%%%%
The weighted Laplacian $\Delta_{\gamma^n_{\alpha}}$ is also called the \emph{Ornstein--Uhlenbeck operator} and is given by 
\[
\Delta_{\gamma^n_{\alpha}} f(x)= \Delta_{\mathbb{R}^n} f(x)- \frac{1}{\alpha^2}\langle x,\nabla_{\mathbb{R}^n} f(x)\rangle
\qquad\text{for\ }f\in C^2(\mathbb{R}^n) \text{\ and\ } x\in \mathbb{R}^n.
\]
For $K=(K_i)_{i=1}^n\in \mathbb{N}_0^n$ and $k\in \mathbb{N}_0$, set
\[
|K|:=\sum_{i=1}^{n} K_i, \qquad
\mathbb{N}_0^n(k):=\left\{ K \in \mathbb{N}^n_0 \ |\ |K|=k\right\}.
\]
The $k$th distinct eigenvalue on $\Gamma^n_{\alpha}$ is given by
\begin{equation}\label{gaussian}
\lambda_k(\Gamma^n_{\alpha})=\frac{k}{\alpha^2}
%%%%
\qquad \text{with multiplicity}\qquad
d_k(n):=\sharp \mathbb{N}_0^n(k)=\binom{n-1+k}{k}
\end{equation}
and $E_k(\Gamma^n_{\alpha})$ is spanned by 
\[
\left\{ x=(x_i)_{i=1}^n\mapsto \prod_{i=i}^n H_{K_i}(\alpha^{-1}x_i)\right\}_{K\in \mathbb{N}_0^n(k)},
\]
where $H_k$ is the $k$th order Hermite polynomial of the form
\begin{equation}\label{hermite}
H_k(r):=(-1)^k e^{\frac{r^2}{2}} \frac{d^k}{dr^k} e^{-\frac{r^2}{2}}.
\end{equation}

An argument similar to the first Dirichlet eigenvalue problem on a ball in a sphere implies that, 
for a first Dirichlet eigenfunction $\psi$ of $-\Delta_{\gamma^n_\alpha}$ on $V_{\alpha R}^n$,
there exists a first Dirichlet eigenfunction~$h$ of $-\Delta_{\gamma^1_\alpha}$ on $V_{\alpha R}^1=(\alpha R, \infty)$ such that 
$\psi(x)=h(x_1)$ on $x=(x_i)_{i=1}^n \in \overline{V_{\alpha R}^n}$, where 
\[
\int_{V_{\alpha R}^n} \psi(x)^2 d\gamma^n_{\alpha}(x)
=
\int_{\alpha R}^{\infty} h(r)^2 d\gamma_{\alpha}^1(r)<\infty.
\]
Moreover, $\lambda(V_{\alpha R}^1, \Gamma_{\alpha}^1)=\lambda(V_{\alpha R}^n, \Gamma^n_{\alpha})$ holds.
%%%%

%%%%%%%%%%%%%%%%%%%%%%%%%%%%%%%%%%%%%%%%%%%%%%%%%%%%
%%%%%%%%%%%%%%%%%%%%%%%%%%
\section{Proof of Theorem~\ref{main1}}
%%%%%%%%%%%%%%%%%%%%%%%%%%%%
%%%
To prove Theorem~\ref{main1}, 
we analyze the composition of $p^N_n$ and homogeneous harmonic polynomials on $\mathbb{R}^{N+1}$.
Given $j\in \mathbb{N}$ and $m\in \mathbb{N}_0$, set
\[
\Delta^j_{\mathbb{R}^n}:=\left( \sum_{i=1}^n \frac{\partial^{2}}{\partial x_i^{2}}\right)^j,\qquad
c_j(m):=-\frac{1}{2j(m+2j-1)},\qquad
C_j(m):=\prod_{l=1}^{j} c_{l}(m) ,
\]
and $\Delta^0_{\mathbb{R}^n}:=\mathrm{id}_{\mathbb{R}^n}, C_0(m):=1$.
For $K=(K_i)_{i=1}^n\in \mathbb{N}_0^n$ and $x=(x_i)_{i=1}^n \in \mathbb{R}^{n}$, 
set
\[
x^K:=\prod_{i=1}^{n} x_i^{K_i},
\]
where by convention $0^0:=1$.
For $t\in \mathbb{R}$, let $[t]$ be the greatest integer less than or equal to~$t$.
\begin{definition}
For $n, N\in \mathbb{N}$ with $n\leq N$, $K\in \mathbb{N}^n_0$ and $a, \alpha>0$, 
define 
\begin{align*}
%%%%
P_{N,n,K}(x,y)&:=\sum_{j=0}^{[|K|/2]} C_j(N-n) |y|_2^{2j} \Delta^{j}_{\mathbb{R}^n} x^K 
&&\text{for\ } (x,y) \in \mathbb{R}^{n} \times \mathbb{R}^{N-n+1},\\
%%%
Q_{N,n,K;a}(x)&:=\sum_{j=0}^{[|K|/2]}C_j(N-n) (a^2-|x|_2^2)^j \Delta^{j}_{\mathbb{R}^n} x^K 
&& \text{for\ } x\in \mathbb{R}^{n},\\
%%%%
Q_{n,K;\alpha}(x)&:=\sum_{j=0}^{[k/2]}(-1)^j \frac{\alpha^{2j}}{2^j j!}\Delta^{j}_{\mathbb{R}^n} x^K
&& \text{for\ } x\in \mathbb{R}^{n}.
\end{align*}
\end{definition}
We easily check that $P_{N,n,K}$ is a homogeneous polynomial on $\mathbb{R}^{N+1}$of degree $|K|$ and $Q_{N,n,K;a},Q_{n,K;\alpha}\in \mathbb{P}(n)$.
All of $\{ P_{N,n,K}\}_{K\in \mathbb{N}^n_0(k)}, \{ Q_{N,n,K;a}\}_{K\in \mathbb{N}^n_0(k)},\{ Q_{n,K;\alpha}\}_{K\in \mathbb{N}^n_0(k)}$ are linearly independent.
It turns out that 
\[
Q_{N,n,K;a}\circ p^N_n =P_{N,n,K}
\qquad
\text{on $\mathbb{S}^N(a)$}.
\]
%%%%%%
%%%%%%%%%%%%%%%%%%%%%%%%%%
\begin{lemma}\label{hori}
For $n,N\in \mathbb{N}$ with $n,2\leq N$ and $k\in \mathbb{N}_0$, let $P$ be a homogeneous harmonic polynomial on $\mathbb{R}^{N+1}$ of degree $k$.
Then $P|_{\mathbb{S}^N(a)}\in E_{k}^{n}(\mathbb{S}^N(a))$ 
if and only if there exists $b_K \in \mathbb{R}$ for each $K\in \mathbb{N}_0^n(k)$ such that 
$P$ is decomposed as 
\[
P
=\sum_{K\in \mathbb{N}_0^{n}(k)} b_K P_{N,n,K}
\qquad
\text{on\ }\mathbb{R}^{N+1}.
\]
\end{lemma}
%%%%%%%%%%%%%%%%%%%%%%%%%%
%%%%%%%%%%%%%%%%%%%%%%%%%%%%%%%%%%%%%%%%%%%%%%%%%%%%
%%%%%%%%%%%%%%%%%%%%%%%%%%%%%%%%%%%%%%%%%%%%%%%%%%%%
\begin{proof}
Let $P$ be a homogeneous harmonic polynomial on $\mathbb{R}^{N+1}$ of degree $k$.

If $P|_{\mathbb{S}^N(a)}\in E_{k}^{n}(\mathbb{S}^N(a))$, then $P$ satisfies
\[
P(x,y)=P(x, |y|_2e_1^{N-n+1})=P(x, -|y|_2e_1^{N-n+1})\qquad \text{for}\ \ (x,y)\in \mathbb{S}^N(a)\subset \mathbb{R}^n \times \mathbb{R}^{N-n+1}.
\]
This implies that
there exists a homogeneous polynomial $Q_{k-2j}$ on $\mathbb{R}^n$ of degree $k-2j$ for each $0\leq j\leq [k/2]$ such that 
\[
P(x,y)=\sum_{j=0}^{[k/2]} |y|_2^{2j} Q_{k-2j}(x) \qquad \text{for\ } (x,y) \in \mathbb{R}^n \times \mathbb{R}^{N-n+1}
\]
%%%%%%%%%%%%%%%%%%%%%%%%%%%%%%%%%%%%%%%%%%%%%%%%%%%%
{(compare with \cite{PS}*{Proposition~3.10}).}
Since $P$ is harmonic, we find that 
\begin{align}
\begin{split}\label{harmonic}
0&=\Delta_{\mathbb{R}^{N+1}} P(x,y)
=\sum_{j=0}^{[k/2]} \left\{ \left( \Delta_{\mathbb{R}^{N-n+1}}|y|_2^{2j}\right) Q_{k-2j}(x)+ |y|_2^{2j} \Delta_{\mathbb{R}^{n}} Q_{k-2j}(x) \right\}\\
%%%
&=\sum_{j=0}^{[k/2]} \left\{ 2j (N-n+2j-1)|y|_2^{2(j-1)} Q_{k-2j}(x) +|y|_2^{2j} \Delta_{\mathbb{R}^{n}} Q_{k-2j}(x) \right\}\\
%%%
&=\sum_{j=1}^{[k/2]}\left\{ \Delta_{\mathbb{R}^{n}} Q_{k-2(j-1)}(x)-\frac{1}{c_j(N-n)} Q_{k-2j}(x)\right\} |y|_2^{2(j-1)} +|y|_2^{2[k/2]} \Delta_{\mathbb{R}^{n}} Q_{k-2[k/2]}(x) \\
&=\sum_{j=1}^{[k/2]}\left\{ \Delta_{\mathbb{R}^{n}} Q_{k-2(j-1)}(x)-\frac{1}{c_j(N-n)} Q_{k-2j}(x) \right\} |y|_2^{2(j-1)},
\end{split}
\end{align}
which implies that
\begin{align*}
Q_{k-2j}(x)
=c_{j}(N-n) \Delta_{\mathbb{R}^{n}} Q_{k-2(j-1)}(x)
=\cdots =C_{j}(N-n) \Delta_{\mathbb{R}^{n}}^{j} Q_{k}(x)
\qquad\text{for\ } 1\leq j \leq [k/2].
\end{align*}
Thus there exists $b_K \in \mathbb{R}$ for each $K\in \mathbb{N}_0^n(k)$ such that 
\[
P(x,y)=
\sum_{j=0}^{[k/2]}C_j(N-n) |y|_2^{2j} \Delta^{j}_{\mathbb{R}^n} \left(\sum_{K\in \mathbb{N}_0^{n}(k)} b_K x^K\right)
=\sum_{K\in \mathbb{N}_0^{n}(k)} b_K P_{N,n,K}(x,y).
\]
Conversely, we observe from~\eqref{harmonic} that $P_{N,n,K}$ is harmonic for each $K\in \mathbb{N}_0^n(k)$.
This complete the proof of the lemma.
\end{proof}
%%%%%%%%%%%%
%%%%%%%%%%%%%%%%%%%%%%%%%%
%%%%%%%%%%%%%%%%%%%%%%%%%%%%%%%%%%%%%%%%%%%%%%%%%%%%
\begin{proof}[Proof of Theorem~\rm{\ref{main1}}]
%%%%%%%%%%%%%%%%%%%%%%%%%%
The relation \eqref{m1} follows from Lemma~\ref{hori}
and \eqref{m2} follows from~\eqref{sphere} together with \eqref{gaussian}, respectively.
%

%%%%
Fix $K\in \mathbb{N}_0^n(k)$.
We prove that 
$\{Q_{N,n,K;a_N}\}_{N}$ converges to $Q_{n,K;\alpha}$ uniformly  on compact sets and strongly in $L^2(\Gamma^n_{\alpha})$ as $N\to \infty$ 
together with $Q_{n,K;\alpha}\in E_{k}(\Gamma^n_{\alpha}) $.
%%%
For $0\leq j\leq [k/2]$, define $q_{N,j}\in \mathbb{P}(n)$ and $q_j\in \mathbb{R}$ by
\[
q_{N,j}(x):=C_j(N-n)(a_N^2-|x|_2^2)^j, \qquad
q_{j}:=(-1)^j \frac{\alpha^{2j}}{2^j j!},
\]
respectively.
Then 
\[
Q_{N,n,K;a_N}(x)=\sum_{j=0}^{[k/2]} q_{N,j}(x)\Delta^{j}_{\mathbb{R}^n} x^K,
\qquad
Q_{n,K;\alpha}(x)=\sum_{j=0}^{[k/2]}q_{j}\Delta^{j}_{\mathbb{R}^n} x^K
\qquad\text{for\ }x\in \mathbb{R}^n. 
\]
Notice that 
$q_{N,0}\equiv 1$ on $\mathbb{R}^n$ and $q_{0}=1$.
For $1 \leq j\leq [k/2] $ and $x\in \mathbb{R}^n$, 
we see that 
\[
q_{N,j}(x) =(-1)^j \prod_{l=1}^j \frac{ a_N^2-|x|_2^2}{2l(N-n+2l-1)}
\xrightarrow{N\to \infty}
 (-1)^j \prod_{l=1}^j \frac{ \alpha^2}{2l} =q_{j}.
\]
Moreover, $\{q_{N,j}\}_{N}$ converges to $q_{j}$ uniformly on compact sets as $N\to \infty$, 
which implies that $\{Q_{N,n,K;a_N}\}_{N}$ converges to $Q_{n,K;\alpha}$ uniformly on compact sets as  $N\to \infty$.
%%%%%
We see that $\{q_{N,j}\}_{N}$ is dominated by $\alpha^{2j} (1+|x|_2^2)^j$ 
hence $\{Q_{N,n,K;a_N}\}_{N}$ is dominated by a certain polynomial on~$\mathbb{R}^n$.
%%%
Since any polynomials on $\mathbb{R}^n$ belongs to $L^2(\Gamma^n_{\alpha})$,
the dominated convergence theorem implies that 
$\{Q_{N,n,K;a_N}\}_{N}$ converges to $Q_{n,K;\alpha}$ strongly in $L^2(\Gamma^n_{\alpha})$ as~$N\to \infty$.
%%%%%

%%% 
%
A direct computation gives
\begin{align*}
\Delta_{\gamma^n_{\alpha}}Q_{n,K;\alpha}(x)
&=
\Delta_{\mathbb{R}^n}Q_{n,K;\alpha}(x)- \frac{1}{\alpha^2}\langle x, \nabla_{\mathbb{R}^n} Q_{n,K;\alpha}(x) \rangle\\
%%%%%
&=
\sum_{j=0}^{[k/2]}q_j\Delta^{j+1}_{\mathbb{R}^n} x^K
-
 \sum_{j=0}^{[k/2]}\frac{q_j}{\alpha^2} \langle x, \nabla_{\mathbb{R}^n}\Delta^{j}_{\mathbb{R}^n} x^K\rangle.
%%%%
\end{align*}
We find that $\Delta^{[k/2]+1}_{\mathbb{R}^n} x^K=0$.
Since 
$\Delta^{j}_{\mathbb{R}^n} x^K$ is a linear combination of $\{ x^{J}\}_{J\in \mathbb{N}_0^n(k-2j)}$
and $\langle x, \nabla_{\mathbb{R}^n} x^J \rangle=|J|x^J$ holds for $J\in \mathbb{N}_0^n$,
%%%
it turns out that 
\[
\langle x, \nabla_{\mathbb{R}^n}\Delta^{j}_{\mathbb{R}^n} x^K\rangle
=(k-2j) \Delta^{j}_{\mathbb{R}^n} x^K,
\]
and consequently 
\begin{align*}
\Delta_{\gamma^n_{\alpha}}Q_{n,K;\alpha}(x)
&=
\sum_{j=0}^{[k/2]-1}q_{j}\Delta^{j+1}_{\mathbb{R}^n} x^K
-
\sum_{j=0}^{[k/2]}\frac{q_j}{\alpha^2}(k-2j)\Delta^{j}_{\mathbb{R}^n} x^K\\
&=
\sum_{j=1}^{[k/2]}\left\{ q_{j-1} -\frac{q_j}{\alpha^2}(k-2j) \right\}\Delta^{j}_{\mathbb{R}^n} x^K\
-\frac{q_0k}{\alpha^2}\Delta^{0}_{\mathbb{R}^n} x^K\\
%%%%%%%
&=
-\sum_{j=1}^{[k/2]}\frac{q_jk}{\alpha^2}\Delta^{j}_{\mathbb{R}^n} x^K\
-\frac{q_0k}{\alpha^2}\Delta^{0}_{\mathbb{R}^n} x^K
%%%%%%%
=-\frac{k}{\alpha^2}Q_{n,K;\alpha}(x).
%%%%
\end{align*}
Thus $Q_{n,K;\alpha}\in E_{k}(\Gamma^n_{\alpha})$ and the proof is complete.
\end{proof}

\begin{remark}
Notice that $\{Q_{N,n,K;a_N}\}_{N}$ does not converge to $Q_{n,K;\alpha}$ uniformly on $\mathbb{R}^n$.
Indeed, if we take $n=1, k=2, I=2$ and $a_N=N^{1/2}$, then 
\begin{align*}
&Q_{N,1,2; \sqrt{N}}(x)=x^2-\frac{N -x^2}{N},
%%%
&&Q_{1,2; 1}(x)=x^2-1 ,
%%%%
%%%%
&& \sup_{x\in \mathbb{R}} \left| Q_{N,1,2; \sqrt{N}}(x)-Q_{1,2; 1}(x)\right|
=\infty.
\end{align*}
%%%%
\end{remark}
%%%%%%%%%%%%%%%%%%%%%%%%%%%%%%%%%%%%%%%%%%%%%%%%
For $(M,\mu)=(\mathbb{S}^N(a), \vol_{\mathbb{S}^N(a)})$ and $\Gamma^n_\alpha$,
it is well-known that all eigenfunctions of $-\Delta_{\mu}$ on~$M$ forms an orthogonal system in $L^2(M,\mu)$.
We denote by $(\cdot, \cdot)_{L^2(M,\mu)}$ and $\|\cdot\|_{L^2(M,\mu)}$ 
the $L^2$-inner product and $L^2$-norm on $(M,\mu)$, respectively.
Let $E^{n}(\mathbb{S}^N(a))$ be the direct sum of $E_{k}^{n}(\mathbb{S}^N(a))$ over $k\in \mathbb{N}_0$
and $E^{n}(\mathbb{S}^N(a))^{\perp}$ its orthogonal complement in $L^2(\mathbb{S}^N(a))$.
The linear space $E^{n}(\mathbb{S}^N(a))$ is spanned by
\[
\left\{P_{N,n,K}|_{\mathbb{S}^N(a)}\right\}_{K\in \mathbb{N}_0^n}.
\]

Set $D^n_a:=\{ x\in \mathbb{R}^n \ |\ |x|_2< a \}$.
%%%
We denote by $\mathbbm{1}_{A}$ the indicator function of a set $A$.
\begin{definition}
Let $\{a_N\}_{N}$ be a sequence of positive real numbers 
such that $\{a_N/\sqrt{N-1}\}_{N}$ converges to a positive real number $\alpha$ as $N\to \infty$.
We define a function $\omega_{a_N,\alpha}$ on $\mathbb{R}^n$ by 
\[
\omega_{a_N,\alpha}(x):=\left(1-\frac{|x|_2^2}{a_N^2}\right)^{\frac{N-n-1}{2}} (2\pi \alpha^2)^{\frac{n}{2}}e^{\frac{|x|^2}{2\alpha^2}} \mathbbm{1}_{D^n_{a_N}}(x).
\]
For $F_N \in E^{n}(\mathbb{S}^N(a_N))$, define a function $f_N$ on $\mathbb{R}^n$ by
\[
f_N(x):=F_N\left(x, \sqrt{a_N^2-|x|_2^2}e_1^{N-n+1} \right) \sqrt{\omega_{a_N,\alpha}}.
\]
We call $f_N$ the \emph{horizontal part} of $F_N$.
\end{definition}
%%%%
It is easy to see that the horizontal part of $P_{N,n,K}|_{\mathbb{S}^N(a_N)}$ is $Q_{N,n,K;a_N} \sqrt{\omega_{a_N,\alpha}}$.
\begin{remark}
For $F_N \in E^{n}(\mathbb{S}^N(a_N))$, put $f(x):=F_N(x, \sqrt{a_N^2-|x|_2^2}e_1^{N-n+1})$ for $x\in D^n_{a_N}$.
Then $f \circ p^N_n=F_N$ holds on $\mathbb{S}^N(a_N)$.
However the converse does not hold, that is, there exist $F\in L^2(\mathbb{S}^N(a))$ and a function $f$ on~$\mathbb{R}^n$ 
such that $f \circ p^N_n=F$ on $\mathbb{S}^N(a_n)$ but $F\notin E^{n}(\mathbb{S}^N(a))$.
%%%
Indeed, for $n=1, N=2$ and $a=\alpha=1$, let
\[
F(z):=z_1^2|_{\mathbb{S}^2(1)}, \qquad f(x):=x^2, \qquad P(z):=z_1^2-z_2^2.
\]
Then $f \circ p^2_1= F$ on $\mathbb{S}^2(1)$ and $P|_{\mathbb{S}^2(1)}\in E_2(\mathbb{S}^2(1))$.
By $E^{1}_2(\mathbb{S}^2(1))=\{ s(3z_1^2-|z|_2^2) \ |\ s\in \mathbb{R}\}$,
$P|_{\mathbb{S}^2(1)}\in E^{1}(\mathbb{S}^2(1))^{\perp}$ follows.
Thus if $F\in E^{1}(\mathbb{S}^2(1))$, then $( F, P|_{\mathbb{S}^2(1)})_{L^2(\mathbb{S}^2(1))}$ should vanish.
However, we compute
\begin{align*}
( F, P|_{\mathbb{S}^2(1)})_{L^2(\mathbb{S}^2(1))}
=
\int_{0}^{2\pi} \int_{0}^\pi \cos^2 \theta_1( \cos^2 \theta_1-\sin^2\theta_1 \cos^2\theta_2)\sin \theta_1d\theta_1 d\theta_2
=\frac{8\pi}{15}.
\end{align*}
\end{remark}

\begin{lemma}\label{perp}
Let $\{a_N\}_{N}$ be a sequence of positive real numbers 
such that $\{a_N/\sqrt{N-1}\}_{N}$ converges to a positive real number $\alpha$ as $N\to \infty$.
Assume $n<N$.
For $F_N \in E^{n}(\mathbb{S}^N(a_N))$ and its horizontal part $f_N$,
it follows that
\[
\|f_N\|^2_{L^2(\Gamma^n_\alpha)}
=\frac{\|F_N\|^2_{L^2(\mathbb{S}^N(a_N))}}{\vol_{\mathbb{S}^{N-n}(a_N)}(\mathbb{S}^{N-n}(a_N))}.
\]
\end{lemma}
\begin{proof}
%
%%%
Set $\Theta:=[0,\pi]^{N-1}\times[0,2\pi]$ and 
define $\zeta=(\xi,\eta):\Theta\to \mathbb{S}^N(1) \subset \mathbb{R}^n\times \mathbb{R}^{N-n+1}$ by
\[
\zeta(\theta)_i=
\begin{cases}
\cos \theta_1&\text{if\ }i=1,\\
\displaystyle \left(\prod_{j=1}^{i-1} \sin \theta_j \right)\cos \theta_i&\text{if\ }2\leq i\leq N,\\
\displaystyle \prod_{j=1}^{N} \sin \theta_j & \text{if\ }i=N+1.
\end{cases}
\]
Moreover, put 
$f(x):=F_N(x, \sqrt{a_N^2-|x|_2^2}e_1^{N-n+1})$ for $x\in D^n_{a_N}$.
Then the change of variables yields
\begin{align*}
\|F_N\|^2_{L^2(\mathbb{S}^N(a_N))}
&=a_N^N\int_{\Theta} F_N(a_N \zeta(\theta))^2 \left( \prod_{i=1}^{N-1} \sin^{N-i} \theta_i \right) d\theta \\
&=a_N^N\int_{\Theta} 
f\left(a_N \xi(\theta)\right)^2 \left( \prod_{i=1}^{N-1} \sin^{N-i} \theta_i \right) d\theta \\
&=2\pi \left(\prod_{i=n+1}^{N-1} \int_0^\pi \sin^{N-i} \theta d \theta\right) \cdot a^{N-n}_N\int_{D^n_{a_N}} 
f\left(x\right)^2
 \left(1-\frac{|x|_2^2}{a_N^2}\right)^{\frac{N-n-1}{2}}dx \\
&=\vol_{\mathbb{S}^{N-n}(a_N)}(\mathbb{S}^{N-n}(a_N))\int_{\mathbb{R}^n} f_N(x)^2d\gamma^n_\alpha(x).
\end{align*}
This concludes the proof of the lemma.
\end{proof}

As a corollary of Theorem~\ref{main1}, 
we show the $L^2$-strong convergence of the heat flow and the Mosco convergence of the Cheeger energy.
%%%
These convergences with respect to the pointed measured Gromov--Hausdorff topology
under the curvature-dimension condition are known. 
For example, see \cite{AH2}*{Theorem~1.5.4}, \cite{GMS}*{Theorems~6.8, 6.11}, \cite{Kazukawa}*{Theorem~1.1} and also 
\cite{AH}*{Theorem~3.4 and Proposition~3.9},\cite{ZZ}*{Theorem~3.8}.
%%%%
The results are concerned with the asymptotic behaviors of Laplacians.
It should be mentioned that, for each \text{$k\in \mathbb{N}$}, 
Peterson--Sengputa~\cite{PS}*{Proposition~5.4} proved the convergence of 
$\Delta_{\mathbb{S}^N(\sqrt{N-1})}$ to the Hermite operator as $N\to \infty$ on  the space of 
homogeneous polynomials  of degree at most $k$, 
and that the  projection of  the Hermite operator onto the first $n$-coordinates is $\Delta_{\Gamma^n_1}$ (see also~\cite{UK}*{Proposition~3}).

%%%%%%%%%%
\begin{corollary}\label{cor}
Let $\{a_N\}_{N}$ be a sequence  of positive real numbers such that 
$\{a_N/\sqrt{N-1}\}_{N}$ converges to a positive real number $\alpha$ as $N\to \infty$.
Let $U_N : [0,\infty) \times \mathbb{S}^N(a_N) \to\mathbb{R}$ denote the solution to the heat equation 
\[
\begin{cases}
\dfrac{\partial}{\partial t} U=\Delta_{\mathbb{S}^N(a_N)} U&\text{in\ } (0,\infty) \times \mathbb{S}^N(a_N),\\
U(0,\cdot)=F_N&\text{in\ } \mathbb{S}^N(a_N),
\end{cases}
\]
where $F_N \in E^{n}(\mathbb{S}^N(a_N))$.
Then $U_N(t,\cdot)\in E^{n}(\mathbb{S}^N(a_N))$ for any $t\geq 0$.
%%%%%

Let $f_N$ and $u_N(\cdot, t)$ be the horizontal part of $F_N$ and $U_N(t,\cdot)$, respectively.
If $\{f_N\}_{N}$ converges to $f_\infty$ weakly in $L^2(\Gamma^n_\alpha)$ as $N\to \infty$,
then $\{u_N(t,\cdot)\}_{N}$ converges to $u_\infty(t,\cdot)$ strongly in $L^2(\Gamma^n_\alpha)$ as $N\to \infty$ for each $t>0$
and $\{u_\infty(t,\cdot)\}_{t\geq 0}$ solves the heat equation
\begin{equation}\label{limit}
\begin{cases}
\dfrac{\partial}{\partial t} u=\Delta_{\gamma^n_\alpha} u&\text{in\ } (0,\infty) \times \mathbb{R}^n,\\
u(0,\cdot)=f_\infty&\text{in\ } \mathbb{R}^n.
\end{cases}
\end{equation}
\end{corollary}
%%%%%%%%%
\begin{proof}
Let $\{\phi_{N,k}\}_{k\in \mathbb{N}}$ be an orthonormal system in $L^2(\mathbb{S}^N(a_N))$
such that each $\phi_{N,k}$ is an eigenfunction of eigenvalue $\lambda_{N,k}$ and 
either $\phi_{N,k}\in E^{n}(\mathbb{S}^N(a_N))$ or $\phi_{N,k}\in E^{n}(\mathbb{S}^N(a_N))^{\perp}$ holds.
It is well-known that $U_N(t,z)$ is given by 
\begin{align*}
U_N(t,z)&=
\sum_{k\in \mathbb{N}}e^{-t\lambda_{N,k}}\left(F_N,\phi_{N,k}\right)_{L^2({\mathbb{S}^N(a_N)})}
 \phi_{N,k}(z). 
\end{align*}
For instance, see~\cite{Ch}*{Section~VI.1}.
We deduce from 
\[
\left(F_N, \phi_{N,k}\right)_{L^2({\mathbb{S}^N(a_N)})}=0
\qquad \text{for\ }\phi_{N,k}\in E^{n}(\mathbb{S}^N(a_N))^{\perp}
\]
that $U_N(t,\cdot)\in E^{n}(\mathbb{S}^N(a_N))$ holds for any $t\geq 0$.

{Without} loss of generality, we may assume that, for each $\phi_{N,k}\in E^{n}(\mathbb{S}^N(a_N))$, there exists $K\in \mathbb{N}^n_0$ such that
\[
\phi_{N,k}
=\frac{P_{N,n,K}|_{\mathbb{S}^N(a_N)}}{ \left\|P_{N,n,K}|_{\mathbb{S}^N(a_N)}\right\|_{L^2(\mathbb{S}^N(a_N))}}.
\]
We shall abbreviate $P_{N,n,K}|_{\mathbb{S}^N(a_N)}$ by $P_{N,n,K}$ when there is no possibility of confusion.
We see that
\begin{align*}
f_N
&=
\sum_{K\in \mathbb{N}_0^{n}}
\frac{\left( F_N,P_{N,n,K}\right)_{L^2(\mathbb{S}^N(a_N))}}
{\left\|P_{N,n,K}\right\|^2_{L^2(\mathbb{S}^N(a_N))}}Q_{N,n,K;a_N} \sqrt{\omega_{a_N,\alpha}},\\
%%%%%%%
u_N(t,\cdot)
&=\sum_{K\in \mathbb{N}_0^{n}}e^{-t\lambda_{|K|}(\mathbb{S}^N(a_N))}
\frac{\left( F_N,P_{N,n,K}\right)_{L^2(\mathbb{S}^N(a_N))}}{\left\|P_{N,n,K}\right\|^2_{L^2(\mathbb{S}^N(a_N))}}Q_{N,n,K;a_N} \sqrt{\omega_{a_N,\alpha}}.
\end{align*}
Similarly, for $f_\infty$ and a solution $u$ to \eqref{limit}, it turns out that
\begin{align*}
f_\infty(x)
&=\sum_{K\in \mathbb{N}_0^{n}}\frac{\left( f_\infty, Q_{n,K;\alpha}\right)_{L^2(\Gamma^n_\alpha)}}{\left\| Q_{n,K;\alpha}\right\|^2_{L^2(\Gamma^n_\alpha)}}Q_{n,K;\alpha},\\
%%%%
u(t,x)
&=
\sum_{K\in \mathbb{N}_0^{n}}e^{-t\lambda_{|K|}(\Gamma^n_\alpha)}
\frac{\left( f_\infty, Q_{n,K;\alpha}\right)_{L^2(\Gamma^n_\alpha)}}{\left\| Q_{n,K;\alpha}\right\|^2_{L^2(\Gamma^n_\alpha)}}Q_{n,K;\alpha}.
\end{align*}
%%%%%%%%
For instance, see~\cite{Bog}*{Theorem~1.4.4}.
As well as the proof of Lemma~\ref{perp}, we find that 
\begin{align*}
\left\| Q_{N,n,K;a_N} (x) \sqrt{\omega_{a_N,\alpha}}\right\|^2_{L^2(\Gamma^n_\alpha)}
&=\frac{\left\|P_{N,n,K}\right\|^2_{L^2(\mathbb{S}^N(a_N))}}{\vol_{\mathbb{S}^{N-n}(a_N)}(\mathbb{S}^{N-n}(a_N))},\\
%%%
\left(f_N, Q_{N,n,K;a_N}\sqrt{\omega_{a_N,\alpha}}\right)_{L^2(\Gamma^n_\alpha)}
&=
\frac{\left( F_N,P_{N,n,K}\right)_{L^2(\mathbb{S}^N(a_N))}}{\vol_{\mathbb{S}^{N-n}(a_N)}(\mathbb{S}^{N-n}(a_N))},
\end{align*}
for $n<N$.
It follows from  the inequality $1-\rho\leq e^{-\rho}$ on $\rho \in\mathbb{R}$ that
\begin{equation}\label{expbdd}
\left(1-\frac{r^2}{a_N^2}\right)^{\frac{N-n-1}{2}}
 \leq
  \exp\left(-\frac{r^2}{a_N^2} \cdot \frac{N-n-1}{2}\right)
\leq
  \exp\left(-\frac{r^2}{4\alpha^2}\right) \quad \text{on }r\in (-a_N,a_N)
\end{equation}
for large enough $N\in \mathbb{N}$.
Then %$\{\omega_{a_N,\alpha}\}_N$  is dominated  by $(2\pi \alpha^2)^{\frac{n}{2}}e^{\frac{|x|^2}{4\alpha^2}}$ and 
\[
\{(2\pi \alpha^2)^{-\frac{n}{4}} Q_{N,n,K;a_N} \sqrt{\omega_{a_N,\alpha}} \}_{N}
\]
is dominated by the product of $\exp(|x|^2/8\alpha^2)$ and a certain polynomial on~$\mathbb{R}^n$,
where the product belongs to $L^2(\Gamma^n_\alpha)$,
hence the sequence converges to $Q_{n,K;\alpha}$ strongly in $L^2(\Gamma^n_\alpha)$ as $N\to\infty$
by the dominated convergence theorem.
This with the weak convergence of $\{f_N\}_{N}$ in $L^2(\Gamma^n_\alpha)$ yields 
\begin{align}
\begin{split}
\label{conv}
\left\| Q_{n,K;\alpha}\right\|^2_{L^2(\Gamma^n_\alpha)}
&=
\lim_{N\to\infty} \left\|(2\pi \alpha^2)^{-\frac{n}{4}}Q_{N,n,K;a_N} \sqrt{\omega_{a_N,\alpha}}\right\|^2_{L^2(\Gamma^n_\alpha)}\\
&=\lim_{N\to\infty}\frac{\left\|P_{N,n,K}\right\|^2_{L^2(\mathbb{S}^N(a_N))}}{\vol_{\mathbb{S}^{N}(a_N)}(\mathbb{S}^{N}(a))},\\
%%%%
\left( f_\infty, Q_{n,K;\alpha}\right)_{L^2(\Gamma^n_\alpha)}
&=
\lim_{N\to\infty}\left(f_N, (2\pi \alpha^2)^{-\frac{n}{4}}Q_{N,n,K;a_N} \sqrt{\omega_{a_N,\alpha}}\right)_{L^2(\Gamma^n_\alpha)}\\
&=
(2\pi \alpha^2)^{\frac{n}{4}}\lim_{n\to\infty}\frac{\left( F_N,P_{N,n,K}\right)_{L^2(\mathbb{S}^N(a_N))}}{\vol_{\mathbb{S}^{N}(a_N)}(\mathbb{S}^{N}(a_N))},
\end{split}
\end{align}
%%%%
where we used the Stirling's approximation to have
\[
\frac{\vol_{\mathbb{S}^{N}(a_N)}(\mathbb{S}^{N}(a_N))}{\vol_{\mathbb{S}^{N-n}(a_N)}(\mathbb{S}^{N-n}(a_N))}
\xrightarrow{N\to\infty}(2\pi \alpha^2)^{\frac{n}{2}}.
\]
The monotonicity of the $L^2$-energy along the heat flow (see~\cite{Ch}*{Proposition~VI.1.1}) provides
\begin{align*}
\sup_{N\in\mathbb{N}}\left\| u_N(t,\cdot)\right\|^2_{L^2(\Gamma_\alpha^n)}
&=
\sup_{N\in\mathbb{N}}\frac{\left\| U_N(t,\cdot) \right\|_{L^2(\mathbb{S}^N(a_N))}^2}{\vol_{\mathbb{S}^{N-n}(a_N)}(\mathbb{S}^{N-n}(a_N))}\\
&\leq
\sup_{N\in\mathbb{N}}\frac{\left\| F_N\right\|_{L^2(\mathbb{S}^N(a_N))}^2}{\vol_{\mathbb{S}^{N-n}(a_N)}(\mathbb{S}^{N-n}(a_N))}
=\sup_{N\in\mathbb{N}}\left\| f_N\right\|^2_{L^2(\Gamma_\alpha^n)}<\infty.
\end{align*}
Then the Banach--Alaoglu theorem implies that there exists a subsequence of $\{ u_{N}(t, \cdot) \}_{N}$, still denoted by
$\{ u_{N}(t, \cdot) \}_{N}$, converging weakly in $L^2(\Gamma_\alpha^n)$.
We denote by $u_\infty(t,\cdot)$ the limit.
We apply the strong convergence of $\{(2\pi \alpha^2)^{-\frac{n}{4}} Q_{N,n,K;a_N}\sqrt{\omega_{a_N,\alpha}}\}_{N}$ again to have
\begin{align*}
\left( u_\infty(t,\cdot), Q_{n,K;\alpha}\right)_{L^2(\Gamma^n_\alpha)}
=&\lim_{N\to \infty}\left( u_{N}(t,\cdot), (2\pi \alpha^2)^{-\frac{n}{4}}Q_{N,n,K;a_{N}}\sqrt{\omega_{a_N,\alpha}}\right)_{L^2(\Gamma^n_\alpha)}\\
%%%%%%%
=&(2\pi \alpha^2)^{\frac{n}{4}}\lim_{N\to \infty} 
e^{-t\lambda_{|K|}(\mathbb{S}^N(a_N))} 
\frac{\left( F_N,P_{N,n,K}\right)_{L^2(\mathbb{S}^N(a_N))}}{\vol_{\mathbb{S}^{N}(a_N)}(\mathbb{S}^{N}(a_N))}\\
=& e^{-t\lambda_{|K|}(\Gamma^n_\alpha)}\left( f_\infty, Q_{n,K;\alpha}\right)_{L^2(\Gamma^n_\alpha)}\\
=&\left( u(t,\cdot), Q_{n,K;\alpha}\right)_{L^2(\Gamma^n_\alpha)},
\end{align*}
which leads to $u_\infty(t,\cdot)=u(t,\cdot)$.
%%%
Thus $\{u_N(t,\cdot)\}_{N}$ converges to $u(t,\cdot)$ weakly in $L^2(\Gamma^n_\alpha)$ as $N\to \infty$ for each $t\geq0$.

For $N\in \mathbb{N}$ and $k\in \mathbb{N}_0$, set
\begin{align*}
%%%
B_{N,k}(t)&:=\sum_{K\in \mathbb{N}_0^{n}, |K|\leq k} e^{-2t\lambda_{|K|}(\mathbb{S}^N(a_N))}
\frac{\left( F_N,P_{N,n,K}\right)^2_{L^2(\mathbb{S}^N(a_N))}}{{\vol_{\mathbb{S}^{N-n}(a_N)}(\mathbb{S}^{N-n}(a_N))}\cdot\left\|P_{N,n,K}\right\|^2_{L^2(\mathbb{S}^N(a_N))}},\\
%%%%
B_{k}(t)&:=\sum_{K\in \mathbb{N}_0^{n}, |K|\leq k} e^{-2t\lambda_{|K|}(\Gamma^n_\alpha)} 
\frac{\left( f_\infty,Q_{n,K;\alpha}\right)^2_{L^2(\Gamma^n_\alpha)}}{\left\|Q_{n,K;\alpha}\right\|^2_{L^2(\Gamma^n_\alpha)}}.
%%%%
\end{align*}
By~\eqref{conv} and Theorem~\ref{main1}, we see that $B_{N,k}(t)\to B_{k}(t)$ as $N\to \infty$ and
\[
\sup_{N\in \mathbb{N},k\in \mathbb{N}_0}B_{N,k}(t)\leq 
\sup_{N\in \mathbb{N}}\lim_{k\to\infty}B_{N,k}(t)
=\sup_{N\in\mathbb{N}}\left\| u_N(t,\cdot)\right\|^2_{L^2(\Gamma_\alpha^n)}
\leq
\sup_{N\in\mathbb{N}}\left\| f_N\right\|^2_{L^2(\Gamma_\alpha^n)}<\infty.
\]
It follows from Dirichlet's test that
\begin{align*}
\left|\lim_{m\to\infty}B_{N,m}(t) -B_{k}(t)\right|
-
\left|B_{k}(t)-B_{N,k}(t)\right|
&\leq
\left|\lim_{m\to\infty}B_{N,m}(t) -B_{N,k}(t)\right|\\
&\leq 2\sup_{m\in \mathbb{N}} \|f_m\|_{L^2(\Gamma^n_\alpha)} e^{-2t\lambda_{k+1}(\mathbb{S}^N(a_N))}. 
\end{align*}
For $t>0$, letting $N\to \infty$ first and then $k\to \infty$ leads to
\[
\lim_{N\to \infty}\left\|u_N(t,\cdot)\right\|^2_{L^2(\Gamma_\alpha^n)}
=\lim_{N\to\infty}\lim_{m\to\infty}B_{N,m}(t) 
=\lim_{k\to\infty}B_{k}(t) 
=\left\| u(t,\cdot)\right\|^2_{L^2(\Gamma_\alpha^n)},
\]
which is the equivalent to the strong convergence of $\{u_N(t,\cdot)\}_{N}$ to $u(t,\cdot)$ in $L^2(\Gamma^n_\alpha)$
as \text{$N\to \infty$.} 
This completes the proof of the corollary.
\end{proof}
%%%%%%%%%%%%%%%%
%
As well as $H^1_0(V_{\alpha R}^n, \gamma^n_\alpha)$,
we define $H^1(M,\mu)$ as the completion of $C_0^\infty(M)$ with respect to the inner product given by
\[
(f_1, f_2)_{H^1(M,\mu)}:=\int_{M} f_1f_2 d\mu
+\int_{M} g(\nabla_{M} f_1,\nabla_{M} f_2) d\mu \qquad\text{for\ } f_1, f_2\in C^\infty_0(M).
\]
For $f\in H^1(M,\mu)$, we write $|\nabla f|_M:=g(\nabla_M f, \nabla_M f)^{1/2}$.
By~\cite{Bog}*{Proposition~1.5.4},
\begin{align*}
H^{1}(\Gamma^n_{\alpha})
&=\left\{ f\in L^2(\Gamma^n_\alpha)\biggm| 
\sum_{K\in \mathbb{N}^n_0} \lambda_{|K|}(\Gamma^n_\alpha)
 \frac{ \left(f, Q_{n,K;\alpha} \right)^2_{L^2(\Gamma^n_\alpha)} }{\| Q_{n,K;\alpha} \|^2_{L^2(\Gamma^n_\alpha)}} <\infty \right\}.
\end{align*}
Similarly, we see that 
\begin{align*}
&H^1(\mathbb{S}^N(a_N)) \cap E^{n}(\mathbb{S}^N(a_N))\\
=&\left\{ F_N \in E^{n}(\mathbb{S}^N(a_N))\biggm| 
\sum_{K\in \mathbb{N}^n_0}
\lambda_{|K|}(\mathbb{S}^N(a_N)) \frac{ \left(F_N, P_{N,n,K} \right)^2_{L^2(\mathbb{S}^N(a_N))} }{\| P_{N,n,K} \|^2_{L^2(\mathbb{S}^N(a_N))}} <\infty \right\}.
%%%%%
\end{align*}
For $f\in H^{1}(\Gamma^n_{\alpha})$ and $F_N\in H^1(\mathbb{S}^N(a_N)) \cap E^{n}(\mathbb{S}^N(a_N))$, 
we find that
\begin{align}\label{cheeger}
%%%%
\begin{split}
\int_{\mathbb{R}^n} |\nabla f|_{\mathbb{R}^n}^2d\gamma^n_\alpha
&=\sum_{K\in \mathbb{N}_0^{n}}\lambda_{|K|}(\Gamma^n_\alpha)\frac{\left( f, Q_{n,K;\alpha}\right)_{L^2(\Gamma^n_\alpha)}^2}{\left\| Q_{n,K;\alpha}\right\|^2_{L^2(\Gamma^n_\alpha)}},\\
\int_{\mathbb{S}^N(a_N)}|\nabla F_N|_{\mathbb{S}^N(a_N)}^2 d\!\vol_{\mathbb{S}^N(a_N)}
&=\sum_{K\in \mathbb{N}_0^{n}} \lambda_{|K|}(\mathbb{S}^N(a_N)) \frac{\left( F_N,P_{N,n,K}\right)_{L^2(\mathbb{S}^N(a_N))}^2}{\left\|P_{N,n,K}\right\|^2_{L^2(\mathbb{S}^N(a_N))}}.
\end{split}
\end{align}
%

%%%%%%%%%%%
\begin{corollary}\label{cor2}
Let $\{a_N\}_{N}$ be a sequence  of positive real numbers such that 
$\{a_N/\sqrt{N-1}\}_{N}$ converges to a positive real number $\alpha$ as $N\to \infty$.
Define the \emph{Cheeger energy} $\mathsf{Ch}_N$ on $H^1(\mathbb{S}^N(a_N)) \cap E^{n}(\mathbb{S}^N(a_N))$ by
\[
\mathsf{Ch}_N(F_N)
:=
\frac{1}{\vol_{\mathbb{S}^{N-n}(a_N)}(\mathbb{S}^{N-n}(a_N))}
\int_{\mathbb{S}^{N}(a_N)}|\nabla F_N|_{\mathbb{S}^N(a_N)}^2 d\!\vol_{\mathbb{S}^N(a_N)}.
\]
For $F_N\in H^1(\mathbb{S}^N(a_N)) \cap E^{n}(\mathbb{S}^N(a_N))$ and its horizontal part $f_N$, 
if $\{f_N\}_{N}$ converges to $f_\infty$ weakly in $L^2(\Gamma^n_\alpha)$ as $N\to \infty$,
then
\begin{align}\label{liminf}
\int_{\mathbb{R}^n} |\nabla f_\infty|_{\mathbb{R}^n}^2d\gamma^n_\alpha \leq \liminf_{N\to \infty} \mathsf{Ch}_N(F_N).
\end{align}
Conversely, for $\widetilde{f}\in H^1(\Gamma^n_\alpha)$,
there exists $\widetilde{F}_N\in H^1(\mathbb{S}^N(a_N)) \cap E^{n}(\mathbb{S}^N(a_N))$ such that 
the sequence of the horizontal parts of $\widetilde{F}_N$ converges to $\widetilde{f}$ strongly in $L^2(\Gamma^n_\alpha)$ as $N\to \infty$ and
\begin{align}\label{limok}
\int_{\mathbb{R}^n} |\nabla\widetilde{f}\,|_{\mathbb{R}^n}^2d\gamma^n_\alpha = \lim_{N\to \infty} \mathsf{Ch}_N(\widetilde{F}_N).
\end{align}
\end{corollary}
%%%%%%%%%
\begin{proof}
By Theorem~\ref{main1}, $ \lambda_{|K|}(\mathbb{S}^N(a_N))\to \lambda_{|K|}(\Gamma^n_\alpha)$ as $N\to\infty$.
Moreover, if $\{f_N\}_{N}$ converges to $f_\infty$ weakly in $L^2(\Gamma^n_\alpha)$ as $N\to \infty$,
then \eqref{conv} holds.
These and \eqref{cheeger} with Fatou's lemma provide \eqref{liminf}.

Conversely, for $\widetilde{f}\in H^1(\Gamma^n_\alpha)$,
we can choose $\widetilde{F}_N\in H^1(\mathbb{S}^N(a_N)) \cap E^{n}(\mathbb{S}^N(a))$ as
\begin{align*}
\widetilde{F}_N
= 
\sum_{K\in \mathbb{N}^n_0}\sqrt{\vol_{\mathbb{S}^{N-n}(a_N)}(\mathbb{S}^{N-n}(a_N))
 \frac{a_N^2%\left\|P_{N,n,K}\right\|^2_{L^2(\mathbb{S}^N(a_N))}
 }{|K|+N-1} 
\cdot
 \frac{( \widetilde{f}, Q_{n,K;\alpha})_{L^2(\Gamma^n_\alpha)}^2}{\alpha^2\left\| Q_{n,K;\alpha}\right\|^2_{L^2(\Gamma^n_\alpha)}}}
\cdot \frac{ P_{N,n,K}|_{\mathbb{S}^N(a_N)} }{\|P_{N,n,K}\|_{L^2(\mathbb{S}^N(a_N))}}.
\end{align*}
In this case, 
the sequence of the horizontal parts of $\widetilde{F}_N$ converges to $\widetilde{f}$ strongly in $L^2(\Gamma^n_\alpha)$ as $N\to \infty$ and
\eqref{limok} holds.
This completes the proof of the corollary.
\end{proof}
%
%
%%%%%%%%%%

%%%%%%%%%%%%%%%%%%%%%%%%%%%%%%%%%%%%%%%%%%%%%%%%
\section{Proof of Theorem~\rm{\ref{main2}}}
%%%%%%%%%%%%%%%%%%%%%%%%%%%%%%%%%%%%%%%%%%%%%%%%
We begin with two lemmas concerning boundedness.
Notice that Stirling's approximation yields
\[
\int_{-a_N}^{a_N} s_N(r)^{\frac{N}{2}-1} dr \xrightarrow{N\to \infty} \sqrt{2\pi} \alpha \qquad
\text{and}\qquad
w_N(r)\xrightarrow{N\to \infty} w_\infty(r)
\quad \text{for each\ }r\in \mathbb{R}.
\]
\begin{lemma}\label{unif}
Let $\{a_N\}_{N}$ be a sequence of positive real numbers such that 
$\{a_N/\sqrt{N-1}\}_{N}$ converges to a positive real number $\alpha$ as $N\to \infty$.
%%%
For $N\in \mathbb{N}$, set 
\[
\varpi_N:=\sup_{r\in (-a_N, a_N)} \frac{w_N(r)}{w_\infty(r)}, \qquad
A_N:=\frac{a_N^2 -\alpha^2(N-2)}{a_N}.
\]
%%%%
Then $\{\varpi_N\}_{N}$ is bounded if and only if 
$\{A_N\}_{N}$ is bounded from above.
\end{lemma}
\begin{proof}
For $r\in (-a_N, a_N)$, we compute
\begin{align*}
\frac{d}{dr} \log \frac{w_N(r)}{w_\infty(r)}
&=-\frac{(N-2)r}{a_N^2-r^2 }+\frac{r}{\alpha^2}
=\frac{r}{\alpha^2(a_N^2-r^2)}\left\{a_N^2- \alpha^2\left(N-2\right)-r^2\right\}.
\end{align*}
In the case of $a_N^2- \alpha^2\left(N-2\right)\leq 0$, we see that 
\[
\varpi_N=\frac{w_N(0)}{w_\infty(0)}=
\left( \int_{-a_N}^{a_N} s_N(r)^{\frac{N}{2}-1} dr\right)^{-1}
\cdot \sqrt{2\pi}\alpha
\xrightarrow{N\to\infty}1.
\]
Thus if all $N \in \mathbb{N}$ except a finite number satisfy $a_N^2- \alpha^2\left(N-2\right)\leq 0$, 
then $\{\varpi_N\}_{N}$ is bounded and $\{A_N\}_{N}$ is bounded from above.

Assume that $a_N^2- \alpha^2\left(N-2\right)> 0$, that is, $A_N>0$ for infinitely many $N \in \mathbb{N}$.
For such~$N$ with $N>2$, we set
\begin{align*}
r_N:=\sqrt{a_N^2- \alpha^2\left(N-2\right)}=\sqrt{a_NA_N}.
\end{align*}
Then we find that $r_N<a_N$ and 
\begin{align*}
\varpi_N&=\frac{w_N(r_N)}{w_\infty(r_N)}=\frac{w_N(-r_N)}{w_\infty(-r_N)},\\
\log \frac{w_N(r_N)}{w_\infty(r_N)} 
&=\log \frac{w_N(0)}{w_\infty(0)}
+
\left(\frac{N}2-1\right) \log\left(1-\frac{r_N^2}{a_N^2}\right)+\frac{r_N^2}{2\alpha^2}\\
%%%
&=\log \frac{w_N(0)}{w_\infty(0)}
+
\left(\frac{N}2-1\right) \left\{\log\left(1-\frac{r_N^2}{a_N^2}\right)+\frac{r_N^2}{a_N^2-r_N^2}\right\} .
\end{align*}
%%%
Since $f_1(s):=\log (1-s)$ is strictly concave on $(-\infty,1)$ and $f_1(0)=0, f_1'(0)=-1$, 
it turns out that 
\[
\log \frac{w_N(r_N)}{w_\infty(r_N)} -\log \frac{w_N(0)}{w_\infty(0)}
=\left(\frac{N}2-1\right) f_1\left(\frac{r_N^2}{a_N^2}\right)+\frac{r_N^2}{2\alpha^2}
<
-\left(\frac{N}2-1\right)\frac{r_N^2}{a_N^2} 
+\frac{r_N^2}{2\alpha^2}
=\frac{A_N^2}{2\alpha^2}.
\]
On the other hand, if we set 
\[
f_2(s):= \log(1-s)+\frac{s}{1-s}\qquad \text{for\ }s\in (-2,1),
\]
then
\[
f_2'(s)= \frac{s}{(1-s)^2}, \qquad
f_2''(s)= \frac{1+s}{(1-s)^3},\qquad
f_2'''(s)= \frac{2(2+s)}{(1-s)^4}
> 0,
\]
consequently,
\begin{align*}
\log \frac{w_N(r_N)}{w_\infty(r_N)} -\log \frac{w_N(0)}{w_\infty(0)}
&=\left(\frac{N}2-1\right) f_2\left(\frac{r_N^2}{a_N^2}\right)
%%%
> \left(\frac{N}2-1\right)\frac12\left(\frac{r_N^2}{a_N^2}\right)^2 f_2''(0)
=\frac{N-2}{4a_N^2}A_N^2.
\end{align*}
Thus $\{\varpi_N\}_{N}$ is bounded if and only if $\{A_N\}_{N}$ is bounded from above.
This completes the proof of the lemma.
\end{proof}
%%%%%%%%%%%%%%%%%%%%%%%%%

\begin{lemma}\label{eigen1}
Let $\{a_N\}_{N}$, $\{\theta_N\}_{N}$ be sequences of real numbers so that 
$a_N>0$ and $\theta_N\in (0,\pi)$ for $N \in \mathbb{N}$.
If  there exist $\alpha>0$ and $R\in \mathbb{R}$ such that 
\[
\lim_{N\to \infty} \frac{a_N}{\sqrt{N-1}}=\alpha, \qquad
\lim_{N\to \infty} a_N \cos \theta_N=\alpha R,
\]
then
\[
\sup_{N\in \mathbb{N}}\lambda(B^N_{a_N\theta_N},\mathbb{S}^N(a_N)) <\infty.
\]
\end{lemma}
\begin{proof}
Assume $n,2 \leq N$.
We see that 
\begin{align*}
\frac{\vol_{\mathbb{S}^N(a_N)}(B^N_{a_N\theta_N})}{\vol_{\mathbb{S}^N(a_N)}(\mathbb{S}^N(a_N))}
%%%%
=\frac{\vol_{\mathbb{S}^N(1)}(B^N_{\theta_N})}{\vol_{\mathbb{S}^N(1)}(\mathbb{S}^N(1))} 
=\frac{\displaystyle\int_0^{\theta_N} \sin^{N-1} \theta d\theta}{\displaystyle\int_{0}^{\pi} \sin^{N-1} \theta d\theta}
=\int_{a_N \cos \theta_N}^{a_N}w_N(r) dr.
%%%
\end{align*}
By an argument similar to \eqref{expbdd} with Stirling's approximation, we find that 
\[
w_N(r) \mathbbm{1}_{(a_N \cos \theta_N,a_N)}(r) \leq \frac{1}{\sqrt{\pi}\alpha} e^{-\frac{r^2}{4\alpha^2}} \qquad\text{on\ }r\in\mathbb{R}.
\]
Then the dominated convergence theorem yields 
\[
\lim_{N\to \infty}\frac{\vol_{\mathbb{S}^N(a_N)}(B^N_{a_N\theta_N})}{\vol_{\mathbb{S}^N(a_N)}(\mathbb{S}^N(a_N))}
=\gamma_{\alpha}^1(\alpha R,\infty)\in (0,1).
\]

Let $\theta_N'\in (0,\pi)$ satisfy 
\[
\frac{\vol_{\mathbb{S}^N(a_N)}(B^N_{a_N\theta'_N})}{\vol_{\mathbb{S}^N(a_N)}(\mathbb{S}^N(a_N))}
=\frac{\vol_{\mathbb{S}^N(1)}(B^N_{\theta'_N})}{\vol_{\mathbb{S}^N(1)}(\mathbb{S}^N(1))} 
=\frac{\gamma_{\alpha}^1(\alpha R,\infty)}{2}.
\]
Then, for all $N \in \mathbb{N}$ except a finite number, 
we see that $\theta_N \geq \theta_N'$ hence
\begin{equation}\label{compari}
\lambda(B^N_{a_N\theta_N},\mathbb{S}^N(a_N))
\leq \lambda(B^N_{a_N\theta'_N},\mathbb{S}^N(a_N))
=\frac{1}{a_N^2} \lambda(B^N_{\theta'_N},\mathbb{S}^N(1))
\end{equation}
by the domain monotonicity of eigenvalues (see~\cite{Ch}*{Section~I.5}).
Since the right-hand side in~\eqref{compari} is bounded due to the monotonicity due to Friedland and Hayman~\cite{FH}*{Theorem~2} as mentioned in the introduction, 
this concludes the proof of the lemma.
\end{proof}
%%%%%%%%%%%%%%%%%%%%%%%%%%%

\begin{proof}[Proof of Theorem~\rm{\ref{main2}}]
Set
\[
\lambda_N:=\lambda(B^N_{a_N\theta_N},\mathbb{S}^N(a_N) ), \qquad I_N:=(a_N \cos \theta_N, a_N), 
\qquad I:=(\alpha R, \infty).
%, \qquad \gamma:=\gamma_{\alpha}^1.
\] 
Then $I_N\subset I$ for any $N\in \mathbb{N}$ by the assumption.
Notice that the density of $\gamma_{\alpha}^1$ with respect to the one-dimensional Lebesgue measure is $w_\infty$.
%%%%%%%%

For a nontrivial solution $\varphi_N$ to~{\rm (D$^N$)} for $(a,\theta)=(a_N,\theta_N)$, 
define $h_N\in C^\infty(I_N)\cap C(\overline{I_N})$ by
\[
h_N(r):=\varphi_N\left( a_N \cdot \arccos\left(\frac{r}{a_N}\right) \right). 
\]
%%%%%%%%%%%
A direct computation provides 
%%%%%%%%
\[
\begin{cases}
\displaystyle
L_N h_N 
=-\lambda_N h_N & \text{in}\ I_N,\\
%%%
h_N>0 & \text{in}\ (a_N \cos \theta_N, a_N],\\
%%%
h_N(a_N \cos \theta_N)=0,
%%%
\end{cases} 
\]
%%%
where $L_N:C^\infty(I_N) \to C^\infty(I_N)$ is defined for $f\in C^\infty(I_N)$ by 
\[
L_N f(r)
:=s_N(r)f''(r)- \frac{Nr}{a_N^2}f'(r) .
\]
%%%%%%%%%%%%%
We can assume that
\[
\int_{I_N} h_N(r)^2 w_N(r)dr
={\displaystyle \int_0^{a_N \theta_N} \varphi_N(\theta)^2 \sin^{N-1}\left(\frac{\theta}{a_N}\right) d\theta}
\cdot \left({\displaystyle \int_{-a_N}^{a_N} s_N(r)^{\frac{N}{2}-1} dr }\right)^{-1}
=1
\]
without loss of generality by~\eqref{int}.
We see that the  first positive Dirichlet eigenfunction  $\phi_N(z):=\varphi_N( d_{\mathbb{S}^N(a_N)}(z,a_N e_1^{N+1}))$ 
of $-\Delta_{\mathbb{S}^N(a_N)}$ on $B^N_{a_N \theta_N}$ satisfies 
\[
\int_{B_{a_N \theta_N}^N} \phi_N(z)^2 d\!\vol_{\mathbb{S}^N(a_N)}(z)=
\vol_{\mathbb{S}^{N-1}(a_N)}(\mathbb{S}^{N-1}(a_N)) \int_{-a_N}^{a_N} s_N(r)^{\frac{N}{2}-1}dr.
\]
An integration by parts leads to
\begin{align*}
\lambda_N
=\lambda_N \int_{I_N} h_N(r)^2 w_N(r) dr
%%%
=- \int_{I_N} ( L_N h_N(r)) h_N(r) w_N(r) dr
=
\int_{I_N} h'_N(r)^2 s_N(r) w_N(r)dr.
\end{align*}
Thus we find that 
\[
\int_{I}h_N(r)^2 \frac{w_N(r)}{w_\infty(r)} \mathbbm{1}_{I_N}(r) d\gamma_{\alpha}^1(r)=1, 
\qquad
\int_{I}h'_N(r)^2s_N(r)\frac{w_N(r)}{w_\infty(r)}\mathbbm{1}_{I_N}(r)d\gamma_{\alpha}^1(r)=\lambda_N.
\]
Moreover, an integration by parts yields 
\begin{align*}
\int_{I_N}r^2 h_N(r)^2 w_N(r) dr
&=
-\frac{a_N^2}{N}\int_{I_N} rh_N(r)^2 \left(s_N(r)w_N(r)\right)' dr\\
&=
\frac{a_N^2}{N}\int_{I_N} h_N(r)^2 s_N(r)w_N(r) dr
+
\frac{2a_N^2}{N}\int_{I_N} rh_N(r) h'_N(r) s_N(r)w_N(r) dr\\
%%%
&\leq \frac{a_N^2}{N}
+
\frac12\int_{I_N} r^2h_N(r)^2w_N(r)dr
+\frac{2a_N^4}{N^2}\int_{I_N}h'_N(r)^2 s_N(r)^2w_N(r) dr\\
%%%%
&\leq \frac{a_N^2}{N} \left(1+\frac{2a_N^2 \lambda_N}{N}\right)
+\frac12\int_{I_N} r^2h_N(r)^2w_N(r)dr,
\end{align*}
where we used Young's inequality in the first inequality.
This ensures that 
\[
\int_{I}r^2h_N(r)^2 \frac{w_N(r)}{w_\infty(r)} \mathbbm{1}_{I_N}(r) d\gamma_{\alpha}^1(r)
=\int_{I_N}r^2 h_N(r)^2 w_N(r) dr
\leq \frac{2a_N^2}{N} \left(1+\frac{2a_N^2 \lambda_N}{N}\right).
\]
Since $\{\lambda_N\}_{N}$ is bounded by Lemma~\ref{eigen1}, 
by the Banach--Alaoglu theorem, there exist a subsequence $\{N(m)\}_{m}$ of $\{N\}_{N}$, 
$h_\infty,\widetilde{h}_\infty \in L^2(I,\gamma_{\alpha}^1)$ and $\lambda\in \mathbb{R}$ such that
\begin{align*}
&h_{N(m)} \sqrt{\frac{w_{N(m)}}{w_\infty}}\mathbbm{1}_{I_{N(m)}}\to h_\infty, \\
&rh_{N(m)}(r)\sqrt{\frac{w_{N(m)}(r)}{w_\infty(r)}}\mathbbm{1}_{I_{N(m)}}\to rh_\infty,\\
&h'_{N(m)} \sqrt{s_{N(m)}\frac{w_{N(m)}}{w_\infty}}\mathbbm{1}_{I_{N(m)}}\to \widetilde{h}_\infty,
\end{align*}
weakly in $L^2(I,\gamma_{\alpha}^1)$ and $\lambda_{N(m)} \to \lambda$ as $m\to\infty$.
We see that $h_\infty$ is nonnegative almost everywhere in $I$.
%%%%%%
For $r\in I_N$, we calculate that 
\begin{align*}
\left( h_N (r)s_N(r)\sqrt{\frac{w_N(r)}{w_\infty(r)}}\right)'
&=
h'_N(r) s_N(r)\sqrt{\frac{w_N(r)}{w_\infty(r)}}
+\frac{r}{2}h_N(r) \sqrt{\frac{w_N(r)}{w_\infty(r)}}
\left(\frac{1}{\alpha^2}s_N(r)-\frac{N+2}{a_N^2}
\right).
\end{align*}
This with the boundedness of 
\begin{align*}
\sup_{N\in \mathbb{N}}\sqrt{s_N(r)}<\infty, \qquad
\sup_{N\in \mathbb{N}}\sup_{r\in I_N}\left|\left(\frac{1}{\alpha^2}s_N(r)-\frac{N+2}{a_N^2}
\right)
\right|<\infty
%%%%%
\end{align*}
implies that 
\[
\left( h_{N(m)}s_{N(m)}\sqrt{\frac{w_{N(m)}}{w_\infty}}\mathbbm{1}_{I_{N(m)}}\right)'
\to \widetilde{h}_\infty
\qquad
\text{weakly in $L^2(I,\gamma_{\alpha}^1)$}
\]
as $m\to\infty$.
By the compact Sobolev embedding on $\Gamma^1_{\alpha}$ (see~\cite{Hooton}*{Theorem~3.1} and also \cite{CZ}*{Section~6}),
we can extract a subsequence, still denoted by $\{N(m)\}_{m}$, such that 
\[
h_{N(m)} s_{N(m)}\sqrt{\frac{w_{N(m)}}{w_\infty}} \mathbbm{1}_{I_{N(m)}} \to h_\infty
\qquad
\text{weakly in $H^1_0(I, \gamma_{\alpha}^1)$ and strongly in $L^2(I, \gamma_{\alpha}^1)$}
\]
as $m\to \infty$, where $h'_\infty= \widetilde{h}_\infty$.
Moreover, we find that 
\[
h_{N(m)} \sqrt{\frac{w_{N(m)}}{w_\infty}} \mathbbm{1}_{I_{N(m)}}(1-s_{N(m)})\to 0 
\qquad \text{strongly in $ L^2(I,\gamma_{\alpha}^1)$ as $m\to \infty$}
\]
and hence
\begin{equation}\label{mass1}
\int_{I} h_\infty(r)^2 d\gamma_{\alpha}^1(r)=1.
\end{equation}
%%%%
% see memo.
%%%%
For $f\in H_0^1(I,\gamma_{\alpha}^1)$, we observe from Lemma~\ref{unif} that 
$\{f' \sqrt{s_Nw_N/w_\infty}\mathbbm{1}_{I_N}\}_N$ converges to $f'$ strongly in $ L^2(I,\gamma_{\alpha}^1)$ as $N\to \infty$
and  compute
\begin{align*}
\int_I h'_\infty(r) f' (r) d\gamma_{\alpha}^1(r)
&=
\lim_{m\to \infty} \int_{I_{N(m)}} h_{N(m)}'(r)s_{N(m)}(r) f'(r) w_{N(m)}(r) dr\\
%%%
&=-\lim_{m\to \infty} \int_{I_{N(m)}} \left( s_{N(m)}(r) h_{N(m)}''(r)-\frac{Nr}{a_{N(m)}^2} h_{N(m)}'(r)\right) f(r) w_{N(m)}(r) dr\\
%%%
&=\lim_{m\to \infty} \int_{I_{N(m)}} \lambda_{N(m)} h_{N(m)} f(r) w_{N(m)}(r) dr\\
%%%%%%
&=\lambda \int_{I} h_\infty(r) f(r) d\gamma_{\alpha}^1(r),
\end{align*}
which ensures that $h_\infty$ is a weak solution to the Dirichlet eigenvalue problem of $-\Delta_{\gamma_{\alpha}^1}$ on~$I$.
By the elliptic regularity theory (see~\cite{GT}*{Theorem~7.10 and Corollary~8.11} for instance),
$h_\infty$ is a Dirichlet eigenfunction of $-\Delta_{\gamma_{\alpha}^1}$ on~$I$ of eigenvalue $\lambda$.
Since $h_\infty$ is nonnegative on $I$, $h_\infty$ is a first positive Dirichlet eigenfunction
and hence $\lambda=\lambda(I,\gamma_{\alpha}^1)=\lambda(V_{\alpha R}^n, \gamma_{\alpha}^n)$.
Thus $ \{\lambda_N\}_{N}$ converges to $\lambda(V_{\alpha R}^n, \gamma_{\alpha}^n)$ as $N\to \infty$.
Moreover, it follows from \eqref{mass1} that $\{h_N s_N\sqrt{w_N/w_\infty}\mathbbm{1}_{I_N}\}_{N}$ converges to $h_\infty$ strongly in $H_0^1(I,\gamma_{\alpha}^1)$
as $N\to \infty$.

If we define 
\[
\psi_N(x):=h_N(x_1)s_N(x_1)\sqrt{\frac{w_N(x_1)}{w_\infty(x_1)}}\mathbbm{1}_{I_N}(x_1),
\qquad
\psi_\infty(x):=h_\infty(x_1),
\qquad
\text{for\ }x=(x_i)_{i=1}^n\in \overline{V_{\alpha R}^n}
\]
then $\psi_N, \psi_\infty \in H_0^1(V_{\alpha R}^n, \gamma^n_{\alpha})$ 
and $\{\psi_N\}_{N}$ converges to $\psi_\infty$ strongly in $H^1_0(V_{\alpha R}^n, \gamma^n_{\alpha})$.
Moreover $\psi_N $ satisfies \eqref{dproj}
and $\psi_\infty$ is the first positive Dirichlet eigenfunction $\psi_\infty$ of $-\Delta_{\gamma_\alpha^n}$ on $V_{\alpha R}^n$
satisfying 
\[
\int_{V_{\alpha R}^n} \psi_\infty(x)^2 d\gamma^n_{\alpha}(x)
=\int_I h_\infty(r)^2d\gamma_{\alpha}^1(r)=1.
\]
Thus the proof is complete.
\end{proof}

\section{Projection of Dirichlet eigenspace on high-dimensional sphere}
We briefly recall some facts of the Dirichlet eigenvalue problem on a ball in a sphere.
See~\cite{Ch}*{Sections~II.5, XII.5} for details.
%%%%

The Dirichlet eigenvalue problem on $B^N_{a\theta}$ in $\mathbb{S}^N(a)$ is reduced to a Sturm--Liouville problem of the form 
\[
\tag{D$_k^N$}
\begin{cases}
\displaystyle
\varphi''(\vartheta)+ (N-1)\frac{\cos (\vartheta/a)}{a \sin (\vartheta/a)}\varphi'(\vartheta)=-\left(\lambda-\frac{\lambda_k(\mathbb{S}^{N-1}(1))}{a^2\sin^2(\vartheta/a)} \right) \varphi(\vartheta)
& \text{in\ $[0,a\theta)$},\\
\varphi(a\theta)=0, &
\end{cases} 
\]
for some $k\in \mathbb{N}_0$.
The collection of $\lambda\in \mathbb{R}$ for which there exists a nontrivial solution $\varphi\in C^2([0,a\theta)) \cap C([0,a\theta])$ to ({D$_k^N$})
consists of a sequence 
\[
0< \lambda_{k,1}(B^N_{a\theta}, \mathbb{S}^N(a))<\lambda_{k,2}(B^N_{a\theta}, \mathbb{S}^N(a))<\cdots <\lambda_{k,j}(B^N_{a\theta}, \mathbb{S}^N(a))<\cdots\uparrow \infty,
\]
and 
$\lambda_{k,j}(B^N_{a\theta}, \mathbb{S}^N(a))$ determines a one-dimensional linear space of solutions for each 
$j\in \mathbb{N}$.
The set of Dirichlet eigenvalues on $B^N_{a\theta}$ is given by
\[
\bigcup_{k\in \mathbb{N}_0, j\in \mathbb{N}} \{ \lambda_{k,j}(B^N_{a\theta}, \mathbb{S}^N(a)) \}.
\]
%%%%
Let $(r,\theta)$ denote polar geodesic coordinates about $ae_1^{N+1}$ in $\mathbb{S}^N(a)$, that is,
\[
\left(r(z), \theta(z)\right):=\left(d_{\mathbb{S}^N(a)}(z, ae_1^{N+1}), \frac{(z_i)_{i=2}^{N}}{ \sqrt{a^2-z_1^2}} \right)
\qquad \text{on\ }z=(z_i)_{i=1}^N \in \mathbb{S}^N(a)\setminus\{\pm ae_1^{N+1}\}.
\]
%
%Notice that 
%\[
%a^2-z_1^2=a^2 \sin^2\frac{r(z)}{a}.
%\]
Given a solution $ \varphi_{N,k,j}$ to {\rm (D$_k^N$)} for $\lambda=\lambda_{k,j}(B^N_{a\theta}, \mathbb{S}^N(a))$ 
and $\Phi \in E_{k}(\mathbb{S}^{N-1}(1))$, 
define a function $\phi_{N,k,j}(\Phi;\cdot)$ on $z\in \overline{B^N_{a\theta}}\setminus\{ a e_1^{N+1} \}$ by 
%%%%
\begin{equation}\label{de}
\phi_{N,k,j}(\Phi;z)
:=\varphi_{N,k,j}(r(z)) \Phi(\theta(z)).
\end{equation}
The function $\phi_{N,k,j}(\Phi;\cdot)$ can be extended to $z=a e_1^{N+1}$ smoothly 
and becomes a Dirichlet eigenfunction on $B^N_{a\theta}$ of eigenvalue $\lambda_{k,j}(B^N_{a\theta}, \mathbb{S}^N(a))$.
Let $E_{k,j}(B^N_{a\theta}, \mathbb{S}^N(a))$ denote the linear space of all Dirichlet eigenfunctions 
$\phi_{N,k,j}(\Phi;\cdot)$ on $B^N_{a\theta}$ of eigenvalue $\lambda_{k,j}(B^N_{a\theta}, \mathbb{S}^N(a))$
given by the form~\eqref{de}.
Then the linear space of all Dirichlet eigenfunctions on $B^N_{a\theta}$ coincides with
\[
\bigoplus_{k\in \mathbb{N}_0, j\in \mathbb{N}} E_{k,j}(B^N_{a\theta}, \mathbb{S}^N(a)).
\]
Notice that
\[
\dim E_{k,j}(B^N_{a\theta}, \mathbb{S}^N(a))=\dim E_{k}(\mathbb{S}^{N-1}(1)). 
\]

A similar argument implies that 
the Dirichlet eigenvalue problem on $V_{\alpha R}^n$ in $\Gamma^n_{\alpha}$ is reduced to a Sturm--Liouville problem of the form
\[
\tag{P\!$_{k}$}
\begin{cases}
\displaystyle
\Delta_{\gamma^1_{\alpha}} h=-\left(\lambda -\frac{k}{\alpha^2} \right) h & \text{in\ $(\alpha R,\infty)$},\\
h(\alpha R)=0,
\end{cases}
\]
for some $k\in \mathbb{N}_0$.
The collection of $\lambda\in \mathbb{R}$ for which there exists a nontrivial solution $h\in C^2((\alpha R,\infty)) \cap C([\alpha R,\infty))$ to~(P\!$_{k}$)
consists of a sequence 
\[
0< \lambda_{k,1}(V_{\alpha R}^n, \Gamma^n_{\alpha})<\lambda_{k,2}(V_{\alpha R}^n, \Gamma^n_{\alpha})<\cdots <\lambda_{k,j}(V_{\alpha R}^n, \Gamma^n_{\alpha})<\cdots\uparrow \infty,
\]
and $\lambda_{k,j}(V_{\alpha R}^n, \Gamma^n_{\alpha})$ determines a one-dimensional linear space of solutions for each $j\in \mathbb{N}$.
The set of Dirichlet eigenvalues on $V_{\alpha R}^n$ is given by
\[
\bigcup_{k\in \mathbb{N}_0, j\in \mathbb{N}} \{ \lambda_{k,j}(V_{\alpha R}^n, \Gamma^n_{\alpha}) \}.
\]
Given a solution $h_{k,j} $ to {\rm (P\!$_k$)} for $\lambda=\lambda_{k,j}(V_{\alpha R}^n, \Gamma^n_{\alpha})$ 
and \text{$K=(K_i)_{i=2}^n \in \mathbb{N}_0^{n-1}(k)$}, define a function 
$\psi_{K,j}$ on $x=(x_i)_{i=1}^n\in \overline{V_{\alpha R}^n}$ by 
%%%%
\begin{equation}\label{deG}
\psi_{K,j}(x)
:=
\begin{cases}
\displaystyle h_{k,j}(x) &\text{if $n =1$}, \\
\displaystyle h_{k, j}(x_1) \prod_{i=2}^n H_{K_i}(\alpha^{-1}x_i) &\text{if $n\geq 2$},
\end{cases}
\end{equation}
where $H_k$ is the $k$th order Hermite polynomial given by \eqref{hermite}.
Then $\psi_{K,j}$ is a Dirichlet eigenfunction on $V_{\alpha R}^n$ of eigenvalue $\lambda_{k,j}(V_{\alpha R}^n, \Gamma^n_{\alpha})$.
Let $E_{k,j}(V_{\alpha R}^n, \Gamma^n_{\alpha})$ denote the linear space of all Dirichlet eigenfunctions $\psi_{K,j}$
on $V_{\alpha R}^n$ of eigenvalue $\lambda_{k,j}(V_{\alpha R}^n, \Gamma^n_{\alpha})$ given by the form \eqref{deG}.
Then the linear space of all Dirichlet eigenfunctions on $V_{\alpha R}^n$ coincides with
\[
\bigoplus_{k\in \mathbb{N}_0, j\in \mathbb{N}} E_{k,j}(V_{\alpha R}^n, \Gamma^n_{\alpha}).
\]
Notice that 
\[
\dim E_{k,j}(V_{\alpha R}^n, \Gamma^n_{\alpha})=d_k(n-1),
\]
where we set $d_{k}(0):=1$.

Let $\Omega=B^N_{a\theta}$ if $M=\mathbb{S}^N(a)$, and $\Omega=V_{\alpha R}^n$ if $M=\Gamma^n_{\alpha}$.
The first Dirichlet eigenvalue $\lambda(\Omega, (M,\mu))$ is $\lambda_{0,1}(\Omega, (M,\mu))$ and 
the multiplicity of $\lambda(\Omega, (M,\mu))$ is $1$.
However, $\lambda_{k,j}(\Omega, (M,\mu))=\lambda_{k',j'}(\Omega, (M,\mu))$ may happen for distinct pairs $(k,j), (k',j')\in \mathbb{N}_0 \times \mathbb{N}$.
%%%%%%%%%%%%%%%%%%%%%%%%%%

As a counterpart of $E_{k}^{n}(\mathbb{S}^N(a))$, we define 
\begin{align*}
E_{k,j}^{n}(B^N_{a\theta}, \mathbb{S}^N(a))
:=
\left\{ \phi \in E_{k,j}(B^N_{a\theta}, \mathbb{S}^N(a))
 \biggm| 
\begin{array}{l}
\phi=\phi_{N,k,j}(\Phi;\cdot) \text{\ defined in~\eqref{de} such that} \\
\phi(x,y)=\phi(x, |y|_2 e_1^{N-n+1}) \text{\ on\  }(x,y)\in B^N_{a\theta}
\end{array}
\right\}.
\end{align*}
By the definition and Lemma~\ref{hori},
we immediately find the following.
%%%%%%%%%%%%%%%%%%%%%%%%%%%%%%%%%%%%%%%%%%%%%%%%%%%%
%%%%%%%%%%%%%%%%%%%%%%%%%%%%%%%%%%%%%%%%%%%%%%%%%%%%
%%%%%%%%%%%%%%%%%%%%%%%%%%%%%%%%%%%%%%%%%%%%%%%%%%%%
%%%%%%%%%%%%%%%%%%%%%%%%%%%%%%%%%%%%%%%%%%%%%%%%%%%%
%%%
\begin{proposition}
Fix $N,j\in \mathbb{N}$ with $2\leq N$, $k\in \mathbb{N}_0$, $a>0$ and $\theta \in (0,\pi)$. 
Let $\phi_{N,k,j}(\Phi;\cdot)$ be a Dirichlet eigenfunction on $B^N_{a\theta}$ of eigenvalue $\lambda_{k,j}(B^N_{a\theta}, \mathbb{S}^N(a))$ defined in~\eqref{de}.

The linear space $E_{k,j}^{1}(B^N_{a\theta}, \mathbb{S}^N(a))$ is nontrivial if and only if $k=0$,
where $E_{0,j}^{1}(B^N_{a\theta}, \mathbb{S}^N(a))$ is spanned by 
\[
\{\phi_{N,0,j}(\mathbbm{1}_{\mathbb{S}^{N-1}(1)};\cdot)\}
\]
and hence $\dim E_{0,j}^{1}(B^N_{a\theta}, \mathbb{S}^N(a))=1$.

For $n\in \mathbb{N}$ with $2\leq n\leq N$, $E_{k,j}^{n}(B^N_{a\theta}, \mathbb{S}^N(a))$ is spanned by 
\begin{align*}
\left\{ \phi_{N,k,j}(P|_{\mathbb{S}^{N-1}(1)};\cdot)\right\}_{P\in E_{k}^{n-1}(\mathbb{S}^{N-1}(1))}.
\end{align*}
%%%%
In the sequel, $\dim E_{k,j}^{n}(B^N_{a\theta}, \mathbb{S}^N(a))=d_k(n-1)$.
\end{proposition}
%%%%
%%%%%%%%%%%
Given $n,N \in \mathbb{N}$ with $2\leq n \leq N$ and $K\in \mathbb{N}_0^{n-1}$, define $R_{N,n,K;a}\in\mathbb{P}(n)$ by
\begin{align*}
&R_{N,n,K;a}(x_1, x'):=\sum_{j=0}^{[|K|/2]} (a^2-|x|_2^2)^j C_j(N-n) \Delta^{j}_{\mathbb{R}^{n-1}} x'^K
\qquad \text{for\ } x=(x_1,x')\in \mathbb{R}\times\mathbb{R}^{n-1}.
\end{align*}
Then, for $z=(x,y)\in B^N_{a\theta}\setminus\{a e_1^{N+1}\} \subset \mathbb{R}^n\times \mathbb{R}^{N-n+1}$,
it turns out that
\[
P_{N-1,n-1,K}(\theta(z))=\left(a^2-x_1^2\right)^{-\frac{|K|}{2}} R_{N,n,K;a}(x)
\]
and hence 
\begin{align*}
 \phi_{N,|K|,j}(P_{N-1,n-1,K}|_{\mathbb{S}^{N-1}(1)}; z)
=& \varphi_{N,|K|,j}(r(z))P_{N-1, n-1,K}(\theta(z))\\
%=&
%\varphi_{N,|K|,j}(r(z))
%\cdot \left(a \sin\frac{r(z)}{a}\right)^{-|K|} R_{N,n,K;a}(x)\\
=&\varphi_{N,|K|,j}\left(a \cdot \arccos\left(\frac{x_1}{a}\right)\right)
\cdot \left(a^2-x_1^2\right)^{-\frac{|K|}{2}} R_{N,n,K;a}(x).
\end{align*}
%
%%%%
To establish a counterpart of Theorem~\ref{main2} for higher Dirichlet eigenvalues and their eigenfunctions,
we may need a uniform estimate of $\lambda_{k,j}(B^N_{a_N\theta_N}, \mathbb{S}^N(a_N))$ with respect to $N\in \mathbb{N}$ as well as Lemma~\ref{eigen1}
and a detailed analysis of $\lambda_{k,j}(V_{\alpha R}^1, \Gamma_\alpha^1)$.

\quad

%\medskip

\begin{ack}
The author would like to thank Shouhei Honda for
fruitful conversations on this topic,  
to Daisuke Kazukawa for his careful reading and advice  to improve Corollary~\ref{cor}, 
and to Tatsuya Tate for his comments and providing relevant references.
She is also grateful to Kazuhiro Ishige and Paolo Salani for providing motivation and encouragement. 
She is pleased to acknowledge the hospitality of 
Dipartimento di Matematica e Informatica``U. Dini", Universit\`a di Firenze
where part of this work was performed.
%%%%%
She also  would like to thank an anonymous referee for their careful reading and comments.

The author was supported in part by JSPS KAKENHI Grant Number %19H05599, 19H01786, 19H01800, 
19K03494
and by International Research Experience and Enhancement for Young Researcher of Tokyo Metropolitan University.
\end{ack}

\end{document}